\documentclass[english,reqno]{amsart}

\usepackage{amsmath, appendix, ulem}
\usepackage[nobysame]{amsrefs}
\usepackage{amssymb, color}
\usepackage[margin=1in]{geometry}
\usepackage{mathrsfs}
\usepackage{graphicx}
\usepackage{subfig}
\usepackage{float}
\usepackage{epsf}
\usepackage{hyperref}
\usepackage{titletoc}

\usepackage{subeqnarray}
\usepackage{cases}
\graphicspath{{../Figures/}}

\numberwithin{equation}{section}

\newtheorem{lem}{Lemma}[section]
\newtheorem{thm}{Theorem}[section]

\newtheorem{cor}[thm]{Corollary}
\theoremstyle{remark}
\newtheorem{rmk}{Remark}[section]

\newtheorem{assum}{Assumption}

\newcommand{\nn}{\nonumber}

\renewcommand{\tilde}{\widetilde}
\renewcommand{\hat}{\widehat}
\renewcommand{\bar}{\overline}

\newcommand{\mc}[1]{\mathcal{#1}}

\newcommand{\EE}{\mathbb{E}}
\newcommand{\RR}{\mathbb{R}}
\newcommand{\NN}{\mathbb{N}}
\newcommand{\PP}{\mathbb{P}}

\newcommand{\OV}{\overline{V}}
\newcommand{\OX}{\overline{X}}
\newcommand{\OY}{\overline{Y}}

\newcommand{\fm}{f^m}
\newcommand{\rom}{\rho^m}
\newcommand{\TE}{\mathcal{E}}

\newcommand{\abs}[1]{\left\lvert#1\right\rvert}

 \allowdisplaybreaks

\author{Cristina Cipriani}
\address{Department of Mathematics, Technical University Munich, Munich, Germany}
\email{cristina.cipriani@ma.tum.de}

\author{Hui Huang}
\address{Department of Mathematics and Statistics, University of Calgary, Calgary, Canada}
\email{hui.huang1@ucalgary.ca}

\author{Jinniao Qiu}
\address{Department of Mathematics and Statistics, University of Calgary, Calgary,  Canada}
\email{jinniao.qiu@ucalgary.ca}

\date{\today}
\thanks{The authors would like to thank Prof. Lorenzo Pareschi for suggesting this problem to them. C. C.  acknowledges  the partial support of the project "Online Firestorms And Resentment Propagation On Social Media: Dynamics, Predictability and Mitigation" of the TUM Institute for Ethics in Artificial Intelligence. H. H. is partially supported by the Pacific Institute for the Mathematical Sciences (PIMS) postdoctoral fellowship. J. Q. is partially supported  by the National Science and Engineering Research Council of Canada (NSERC) and by the start-up funds from the University of Calgary. }
\begin{document}
\title[Zero-inertia limit for PSO]{Zero-inertia limit: From Particle Swarm Optimization to Consensus-Based Optimization}
\maketitle
\begin{abstract}
Recently a continuous description of the particle swarm optimization (PSO) based on a system of stochastic differential equations was proposed by Grassi and Pareschi in \cite{grassi2020particle} where the authors formally showed the link between PSO and the consensus-based optimization (CBO) through zero-inertia limit.  This paper  is devoted to solving this theoretical open problem proposed in  \cite{grassi2020particle} by providing a rigorous derivation of CBO from PSO through the limit of zero inertia, and a quantified convergence rate is obtained as well. The proofs are based on a probabilistic approach by investigating both the weak and strong convergence of the corresponding stochastic differential equations (SDEs) of Mckean type in the continuous path space and the results are illustrated with some numerical examples.
\end{abstract}

\section{Introduction}
Over the last decades, large systems of interacting particles are widely used in the investigation of complex systems that
model collective behavior (or swarming), an area that has attracted a great deal of attention; see for instance \cite{carrillo2010asymptotic,bellomo2013microscale,ha2008particle,motsch2014heterophilious} and references therein.
Such complex systems frequently appear in modeling phenomena such as biological swarms \cite{cucker2007emergent}, crowd  dynamics \cite{bellomo2011modeling}, self-assembly of nanoparticles \cite{holm2006formation}, and opinion formation \cite{motsch2014heterophilious}. In the field of global optimization, similar particle models are also used in \textit{metaheuristics} 
\cite{Aarts:1989:SAB:61990,Back:1997:HEC:548530,Blum:2003:MCO:937503.937505,Gendreau:2010:HM:1941310}, which provide empirically robust solutions to tackle hard optimization problems with fast algorithms. Metaheuristics are  methods that orchestrate an interaction between local improvement procedures and global/high level strategies, and combine random and deterministic decisions,
to create a process capable of escaping from local optima and performing a robust search of a solution space. In the sequel, we consider the following optimization problem
\begin{equation}\label{optimization}
x^*\in\mbox{argmin}_{x\in \RR^d}\TE(x)\,,
\end{equation}
where $\TE(x):~\RR^d\to \RR$ is a given continuous cost function.

One noble example of metaheuristics is the so-called Particle Swarm Optimization (PSO), which was initially introduced to model the intelligent collective behavior of complex biological systems such as flocks of birds or schools of fish \cite{kennedy1995particle,kennedy1997particle,shi1998modified}, and it is now widely recognized as an efficient method for tackling complex optimization problems \cite{poli2007particle,lin2008particle}. Certain convergence and stability analysis of PSO may be found, for instance, in \cite{bruned2018weak,van2010convergence,schmitt2015particle,poli2009mean} and readers would be referred to \cite{zhang2015comprehensive} for a review on the PSO method and its variants and applications.
The PSO method solves optimization problem \eqref{optimization} by considering a group of candidate solutions, which are represented by particles. Then the algorithm moves those particles in the search space according to certain mathematical relationships on the particle position and velocity. Each particle is driven to its best known local location, which is updated once the particles find better positions.
However,  the mathematical understanding of PSO is still in its infancy. Recently Grassi and Pareschi \cite{grassi2020particle} took a significant first step towards a mathematical theory for PSO based on a continuous description in the form of a system of stochastic differential equations:
\begin{align}\label{PSO}
\begin{cases}
dX_t^{i,m}= V_t^{i,m}d t, \\
dV_t^{i,m}=-\frac{\gamma}{m}V_t^{i,m}dt+\frac{\lambda}{m}(X_t^\alpha(\rho^{N,m})-X_t^{i,m})dt+\frac{\sigma}{m}D(X_t^\alpha(\rho^{N,m})-X_t^{i,m})dB_t^i,\quad i=1,\cdots,N\,,
\end{cases}
\end{align}
where the $\RR^d$-valued functions $X_t^{i,m} $ and $V_t^{i,m}$ denote the position and velocity of the $i$-th particle at time $t$, $m>0$ is the inertia weight, $\gamma=1-m\geq0$ is the friction coefficient,  $\lambda>0$ is the acceleration coefficient, $\sigma>0$ is the diffusion coefficient,  and $\{(B_t^i)_{t\geq0}\}_{i=1}^N$ are $N$ independent $d$-dimensional Brownian motions. 
We also use the notations for the diagonal matrix
$$D(X_t):=\mbox{diag}\{(X_t)_1,\dots,(X_t)_d\}\in\RR^{d\times d}\,,$$ 
where $(X_t)_k$ is the $k$-th component of $X_t$, and  the weighted average is given by
\begin{equation}\label{regularizer}
{X}_{t}^{\alpha}(\rho^{N,m}):=\frac{\int_{\RR^d}x\omega_{\alpha}^{\mc{E}}(x)\rho^{N,m}(t,dx)}{\int_{\RR^d}\omega_{\alpha}^{\mc{E}}(x)\rho^{N,m}(t,dx)}
\end{equation}
with  the empirical measure $ \rho^{N,m}(t,dx):=\frac{1}{N}\sum_{i=1}^{N}\delta_{X_t^{i,m}}(dx)$. So we can rewrite
\begin{equation}\label{noise}
D(X_t^\alpha(\rho^{N,m})-X_t^{i,m})dB_t^i=\sum_{k=1}^d(X_t^\alpha(\rho^{N,m})-X_t^{i,m})_kd(B_t^i)^ke_k\,,
\end{equation}
where $e_k$ is the unit vector in the $k$-th dimension for $k=1,\dots,d$. Noise of the form \eqref{noise} is called the anisotropic component-wise noise, which was used in \cite{CJ,fornasier2021anisotropic} to remove the dimensionality dependence for the CBO method.
Furthermore, the initial data $(X_0^i,V_0^i)_{i=1}^N$ are independent and identically distributed (i.i.d.) with  the common distribution $f_0\in \mc{P}_4(\RR^{2d})$, where $\mc{P}_4(\RR^{2d})$ denotes the space of probability measures with finite fourth moment, endowed with the Wasserstein distance \cite{ambrosio2008gradient}. The choice of the weight function $$\omega_\alpha^\TE(x):=\exp(-\alpha\TE(x))$$ comes from  the  well-known Laplace's principle \cite{miller2006applied,Dembo2010}, a classical asymptotic method for integrals, which states that for any probability measure $\rho\in\mc{P}( \RR^d )$, there holds
\begin{equation}\label{lap_princ}
\lim\limits_{\alpha\to\infty}\left(-\frac{1}{\alpha}\log\left(\int_{ \RR^d }\omega_\alpha^\TE(x)\rho(dx)\right)\right)=\inf\limits_{x \in \rm{supp }(\rho)} \TE(x)\,.
\end{equation}
Thus for $\alpha$ large enough, one expects that $${X}_{t}^{\alpha}(\rho^{N,m})\approx\mbox{argmin }\{\TE(X_t^{1,m}),\dots,\TE(X_t^{N,m})\}\,,$$
which means that ${X}_{t}^{\alpha}(\rho^{N,m})$ is a global best location at time $t$.

\begin{figure}[tb]
	\includegraphics[scale=0.49]{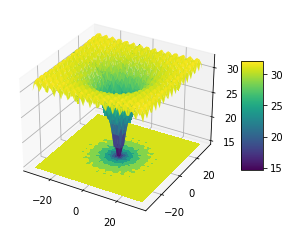}\,
	\includegraphics[scale=0.4]{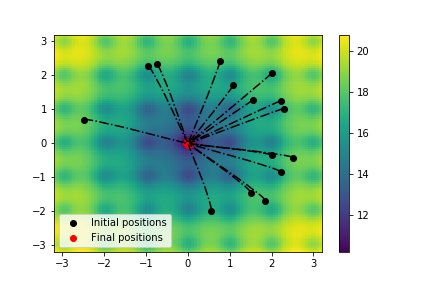}\,
	\caption{Left: the Ackley function for $d=2$ with the unique  global minimum at the point $x^*=(0,0)^T$. Right: Particles trajectories of  the PSO model \eqref{PSO} during the simulation for the 2-$d$ Ackley function with the global minimizer $x^*$. The simulation parameters are: time discretization $0.01$, number of particles $10^{3}$, $\lambda = 1$, $\sigma = \frac{1}{\sqrt{3}}$, $\alpha= 30$, $m= 0.1$. The initial data are sampled from a normal bi-dimensional distribution.}
	\label{fg:ackley}
\end{figure}

\begin{figure}[tb]
\includegraphics[scale=0.32]{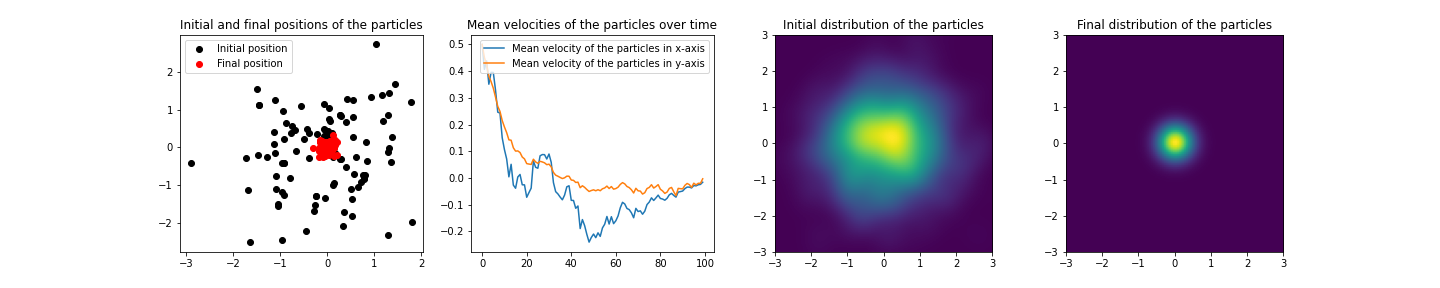}\,
	\caption{Application of the PSO  dynamics \eqref{PSO} to the 2-$d$ Ackley function $\TE(x)$ with the global minimizer $x^*=(0,0)^T$.
		Particles initially have a normal distribution around $x^*$. Then all particles converge to one point,  the global minimizer $x^*$,  and they stop moving eventually, i.e. velocity converges to zero. The simulation parameters are the ones described below Figure \ref{fg:ackley}.}
	\label{fg:particle}
\end{figure}

Before starting our analysis of the PSO dynamics \eqref{PSO}, let us illustrate
numerically the behavior of the dynamics for the benchmark Ackley function
$$\mathcal{E}(x)=-20 \exp \left(-\frac{0.2}{\sqrt{d}}\left|x-x^{*}\right|\right)-\exp \left(\frac{1}{d} \sum_{k=1}^{d} \cos \left(2 \pi \left(x_k-x_{k}^{*}\right)\right)\right)+e+20$$
in the case of $d=2$, and with the global minimizer $x^*=(0,0)^T$. 
In Figures \ref{fg:ackley} and \ref{fg:particle}, we initialize the particles with a normal distribution around $x^*$ and then apply a discretization scheme (which will be explained in Section \ref{numerics}) to the system \eqref{PSO}. We can see that all the particles successfully find the global minimizer $x^*$, and particles' velocity converges to zero.

As it has been shown in \cite{grassi2020particle}, in the zero-inertia limit ($m\to 0$), one may expect to obtain  the recent developed Consensus-Based Optimization (CBO) dynamics \cite{carrillo2018analytical,ha2019convergence,PTTM,fornasier2021consensus} satisfying
\begin{equation}\label{CBO}
dX_t^i= \lambda(X_t^{\alpha}(\rho^N)-X_t^i)dt +\sigma D(X_t^{\alpha}(\rho^N)-X_t^i)dB_t^i,\quad i=1,\cdots,N\,,
\end{equation}
where $$
{X}_{t}^{\alpha}(\rho^{N}):=\frac{\int_{\RR^d}x\omega_{\alpha}^{\mc{E}}(x)\rho^{N}(t,dx)}{\int_{\RR^d}\omega_{\alpha}^{\mc{E}}(x)\rho^{N}(t,dx)}
\mbox{ with } \rho^{N}(t,dx):=\frac{1}{N}\sum_{i=1}^{N}\delta_{X_t^{i}}(dx)\,.$$
It has been proved that CBO is a powerful and robust method to solve many interesting non-convex high-dimensional optimization problems in machine learning \cite{CJ}. By now, CBO methods have also been generalized  to optimization over manifolds \cite{FHPS1,FHPS2,kim2020stochastic,fornasier2021anisotropic}. The objective of the present paper is to complete a theory gap suggested in \cite{grassi2020particle} by providing a rigorous proof of the zero-inertia limit. 

On the one hand, as $N\to\infty$, the mean-field limit results (see \cite{bolley2011stochastic,huang2020mean,carrillo2019propagation,sznitman1991topics,jabin2017mean,huang2021note,huang2021m} for instance) indicate that our PSO dynamics \eqref{PSO} converges to the solutions of the following  mean-field  nonlinear Mckean systems:
\begin{subequations}\label{MVeq}
	\begin{numcases}{}
	d\OX_t^m= \OV_t^mdt, \label{eqX}\\
	d\OV_t^m=-\frac{\gamma}{m}\OV_t^mdt+\frac{\lambda}{m}(X_t^{\alpha}(\rho^m)-\OX_t^m)dt+\frac{\sigma}{m}D(X_t^{\alpha}(\rho^m)-\OX_t^m)dB_t\,, \label{eqV}
	\end{numcases}
\end{subequations}
where 
\begin{equation}\label{14}
X_t^{\alpha}(\rho^m)=\frac{\int_{\RR^d}x\omega_{\alpha}^{\mc{E}}(x)\rho^m(t,dx)}{\int_{\RR^d}\omega_{\alpha}^{\mc{E}}(x)\rho^m(t,dx)}, \quad \rho^m(t,x)=\int_{\RR^d}f^m(t,x,dv)\,,
\end{equation}
and the initial data $(\OX_0,\OV_0)$ is the same as in \eqref{PSO}. Here  $f^m(t,x,v)$ is  the distribution of $(\OX_t^m,\OV_t^m)$ at time $t$ , which makes the set of equations \eqref{MVeq} nonlinear.  We refer to \cite{huang2021note} for a proof the well-posedness of the PSO particle system \eqref{PSO} and its mean-field dynamic \eqref{MVeq}. A direct application of the It\^{o}-Doeblin formula yields that the law $f_t^m:=f^m(t,\cdot,\cdot)$ at time $t$ is a weak solution to the following  nonlinear Vlasov-Fokker-Plank equation
\begin{equation}\label{PSO eq}
\partial_{t} \fm_t+v \cdot \nabla_{x} \fm_t=\nabla_{v} \cdot\left(\frac{\gamma}{m} v \fm_t+\frac{\lambda}{m}\left(x-{X}_t^{\alpha}(\rom)\right) f_t+\frac{\sigma^{2}}{2 m^{2}} \left[D\left(x-{X}_t^{\alpha}(\rom)\right)\right]^{2} \nabla_{v} \fm_t\right),
\end{equation}
with the initial data $\fm_0(x,v)=\mbox{Law}(\OX_0,\OV_0)$.  On the other hand, taking $N\to \infty$ in \eqref{CBO} leads to the mean-field CBO dynamic of the form
\begin{align}\label{MVCBO}
d\OX_t=\lambda(X_t^{\alpha}(\rho)-\OX_t)dt +\sigma D(X_t^{\alpha}(\rho)-\OX_t)dB_t
\end{align}
with $\rho_t=\mbox{Law}(\OX_t)$ satisfying the corresponding CBO equation
 \begin{equation}\label{CBOeq}
\partial_{t} \rho_t+\lambda\nabla_{x} \cdot (\rho_t({X}_t^{\alpha}(\rho)-x) )=\frac{\sigma^{2}}{2} \sum_{j=1}^{d} \frac{\partial^{2}}{\partial x_j^{2}}\left(\rho_t\left(x_{j}-({X}_{t}^{\alpha}(\rho))_j\right)^{2}\right)\,.
\end{equation}

In this paper, we prove that  in the zero-inertia limit, as $m\to 0^+$, the processes $\{\OX^m\}$ satisfying SDEs \eqref{MVeq}  converge to the solution  $\OX$ to the SDE \eqref{MVCBO} in the continuous path space $\mc{C}([0,T];\RR^d)$. A convergence rate is obtained and the generalization to the case with memory effects is also addressed. This is related to the study of the overdamped limit \cite{kramers1940brownian,choi2020quantified,duong2017variational}, or large friction limit \cite{carrillo2020quantitative,jabin2000macroscopic,fetecau2015first} for Vlasov type equations.  However, the nonlinear term $X_t^\alpha(\rho^m)$ here makes our model very different from theirs, which is nonstandard in the literature. Moreover, all of those results mentioned earlier are obtained through the investigation of PDEs like \eqref{PSO eq} and \eqref{CBOeq}, while in the present paper we adopt a probabilistic approach  by investigating the convergence of the non-Markovian stochastic processes $\{\OX^m\}$ satisfying the SDE \eqref{MVeq} to the solution $\{\OX\}$ to SDE \eqref{MVCBO})  in the continuous path space.

The rest of the paper is organized as follows: In Section 2 we verify the tightness of the PSO model \eqref{MVeq} through Aldous criteria, which allows us to obtain the zero-inertia limit from the PSO model \eqref{MVeq} towards the CBO model \eqref{MVCBO} as $m\to 0^+$; see Theorem \ref{thmzero-limit}. Then in Section 3 we generalize the result to the PSO model with memory effects of the local best positions. Lastly we conclude this paper in Section 4 by reporting a few instructive numerical experiments that aim to validating the zero-inertia limit.

\section{Zero-inertia limit}
Throughout this work, the letter $C$ denotes a generic constant whose value may vary from line to line and its dependence on certain model parameters will be specified whenever needed. We start this section with the standing assumption on the cost function $\TE$.
\begin{assum}\label{asum}
	The given cost function $\TE:\RR^d\rightarrow \RR$ is locally Lipschitz continuous and satisfies the properties:
	\begin{itemize}
		\item[(1)]  There exists some constant $L>0$ such $|\TE(x)-\TE(y)|\leq L(|x|+|y|)|x-y|$ for all $x,y\in\RR^d$;
		\item[(2)]	$\TE$ is uniformly bounded, i.e. $-\infty < \underline{\TE}:=\inf \TE \leq \TE\leq \sup \TE=:\overline \TE < + \infty$,  and define $C_{\alpha,\TE}:=e^{\alpha(\overline \TE-\underline{\TE})}$\,.
	\end{itemize}
\end{assum}

The following theorem gives the well-posedness of the mean-field PSO dynamics  \eqref{MVeq} whose proof is  analogous to \cite[Theorem 2.3]{huang2021note} and \cite[Theorem 3.1]{carrillo2018analytical}, and thus omitted.
\begin{thm}\label{thm-huang2021note}
Let Assumption \ref{asum} hold. If $(\OX^m_0,\OV^m_0)=(\OX_0,\OV_0)$ is distributed according to $f_0$ with $f_0\in\mc{P}_4(\RR^{2d})$, then for each $m\in(0,1]$ and $T>0$, the nonlinear SDE \eqref{MVeq} admits a unique strong solution up to time $T$ with the initial data $(\OX^m_0,\OV^m_0)$, and it holds further that
\begin{equation}\label{secmen}
\sup\limits_{t\in[0,T]}\EE\left[|\OX_t^m|^4+|\OV_t^m|^4\right]\leq e^{CT} \cdot \EE\left[ |\OX_0|^4+|\OV_0|^4\right]\,,
\end{equation}
where $C$ depends only on $\lambda,m,\sigma$, and $C_{\alpha,\mc{E}}$.
\end{thm}

 Solving \eqref{eqV} for $\OV_t^m$ gives  
 \begin{equation*}
\OV_t^m=e^{-\frac{\gamma}{m}t}\left(\OV_0+\frac{\lambda}{m}\int_0^te^{\frac{\gamma s}{m}}(X_s^{\alpha}(\rho^m)-\OX_s^m)ds+\frac{\sigma}{m}\int_0^te^{\frac{\gamma s}{m}}D(X_s^{\alpha}(\rho^m)-\OX_s^m)dB_s\right)\,,
\end{equation*}
which implies that
\begin{align}\label{onlyX}
\OX_t^m&=\OX_0+\int_0^t\OV_\tau d\tau
=\OX_0+\int_0^t  e^{-\frac{\gamma}{m}\tau}\OV_0d\tau
+\frac{\lambda}{m}\int_0^te^{-\frac{\gamma}{m}\tau} \int_0^\tau e^{\frac{\gamma}{m}s}(X_s^{\alpha}(\rho^m)-\OX_s^m)dsd\tau\notag\\
& \quad +\frac{\sigma}{m}\int_0^te^{-\frac{\gamma}{m}\tau} \int_0^\tau e^{\frac{\gamma}{m}s}D(X_s^{\alpha}(\rho^m)-\OX_s^m)dB_s d\tau\,.
\end{align}
Then $\OX_t^m$ has the law $\rho^m_t$ for each $t\geq 0$. 

For each $k\in\mathbb N^+$, denote by $\mc{C}([0,T];\RR^k)$ the space of all $\RR^k$-valued continuous functions on $[0,T]$ equipped with the usual uniform norm $\|\cdot\|_0$. 
Each continuous stochastic process $\OX^m$ may be seen as a $\mc{C}([0,T];\RR^d)$-valued random function and it induces a probability measure (or law, denoted by $\rho^m$) on $\mc{C}([0,T];\RR^d)$. We shall use the convergence in the space of probability measures on $\mc{C}([0,T];\RR^d)$. In what follows,  we write $\OX^m \rightharpoonup \OX$ or $\rho^m \rightharpoonup \rho$ with $\rho$ being the law of $\OX$,  if $\left\{\rho^m\right\}_{m>0}$, as a sequence of probability measures,  converges weakly to $ \rho$, i.e.,
for each bounded continuous functional $\Phi$ on $\mc{C}([0,T];\RR^d)$ , there holds $\lim_{m\rightarrow 0^+}\EE\left[\Phi(\OX^m)\right]= \EE\left[\Phi(\OX)\right]$. The weak convergence $\OX^m \rightharpoonup \OX$ is stronger than,  and actually implies the convergence of $\{\rho^m_t \}_{m>0}$ to $ \rho_t$ with $\rho_t$ being the law of $\OX_t$ for each $t\geq 0$, while the converse need not hold. Moreover, due to the separability and completeness of the space $\mc{C}([0,T];\RR^d)$,  Prohorov's theorem implies that the relative compactness is equivalent to the tightness; see \cite{billingsley1999convergence} for more details.

The proof of zero-inertia limit will proceed in two steps:
\begin{itemize}
	\item The tightness of the sequence of probability distributions $\{\rho^m\}_{0< m\leq 1}$ of $\{\OX^m\}_{0< m\leq 1}$ is justified by using Aldous tightness criteria.
	\item We will check that all the limit points of $\{\OX^m\}_{0< m\leq 1}$ as $m\to 0$ satisfy mean-field CBO dynamic \eqref{MVCBO} which in fact admits a unique strong solution.
\end{itemize}
For the sake of completeness, we recall a result concluded directly from the Aldous tightness criteria \cite[Theorem 4.5]{jacod2002limit}; it might have been existing somewhere which, however, we did not find, so its straightforward proof is sketched below.
\begin{lem}\label{lemAldous}
Let $\{X^n\}_{n\in \NN}$ be a sequence of random variables defined on a probability space $(\Omega,\mc{F},\PP)$ and valued in $\mc{C}([0,T];\RR^d)$. The sequence of probability distributions $\{\mu_{X^n}\}_{n\in \NN}$  of $\{X^n\}_{n\in \NN}$ is tight if the following two conditions hold.

$(Con 1)$ For all $t\in [0,T]$, the set of distributions of $X_t^n$, denoted by $\{\mu_{X_t^n}\}_{n\in\NN}$, is tight in $\RR^d$.

$(Con 2)$ For all $\varepsilon>0$, $\eta>0$, there exists $\delta_0>0$ and $n_0\in\NN$ such that for all $n\geq n_0$ and for all $(\sigma(X^n_s;s\in[0,t]))_{t\geq 0}$-stopping times $\beta$ satisfying $0\leq \beta+\delta_0\leq T$, it holds that
\begin{equation}
\sup_{\delta\in[0,\delta_0]}\PP\left(|X^n_{\beta+\delta}-X^n_{\beta}|\geq \eta\right)\leq \varepsilon\,.
\end{equation}
 \end{lem}
 
 \begin{proof}[Sketched proof]\label{rmk-aldous}
 First, we note that the Aldous tightness criteria is normally stated for the tightness of stochastic processes valued in the space of all c\`{a}dl\`{a}g (right continuous with left limits) functions under the so-called Skorokhod topology (see \cite[Theorem 1]{aldous1978stopping}, \cite[Theorem 16.10, Page 178]{billingsley1999convergence} or \cite[Theorem 4.5, Page 356]{jacod2002limit}).  In fact, by the assertion at the beginning of the proof of \cite[Theorem 1]{aldous1978stopping}, the condition $(Con 2)$ implies Hypothesis (A) therein (see \cite[Pages 335 \& 338]{aldous1978stopping}). Therefore, applying \cite[Theorem 1]{aldous1978stopping} directly gives the tightness of $\{X^n\}_{n\in \NN}$ in the space of all c\`{a}dl\`{a}g functions under the Skorokhod topology.

Noteworthily, the condition $(Con 1)$ implies that
 \begin{align}
 \lim_{\eta\rightarrow \infty}\limsup_{n\rightarrow \infty}\PP(|X_t^n|> \eta)=0, \quad \forall\, t\in[0,T],
 \label{B-condition1}
 \end{align}
 Indeed, for each $t\in [0,T]$ and $\epsilon>0$, the tightness in condition $(Con 1)$ indicates that there exists a compact set $\mathcal K_{\epsilon}\subset \RR^d$ such that 
$ \sup_{n\in\mathbb N}\PP (X_t^n\notin\mathcal K_{\epsilon})< \epsilon$. Choosing $\eta_{\epsilon}>0$ to be so big that  $\mathcal K_{\epsilon}\subset \{x\in\mathbb R^d:\, |x|\leq \eta_{\epsilon}\}$, we have $\sup_{n\in\mathbb N}\PP (|X_t^n|>\eta_{\epsilon})< \epsilon$, which by the arbitrariness of $t$ and $\epsilon$ implies 
 \eqref{B-condition1}.
 
 Further, by \cite[Proposition 3.26 (i) \& (iii), Page 351]{jacod2002limit}, because each $X^n$ herein is valued in $\mc{C}([0,T];\RR^d)$ having continuous trajectories, we have the $\mc C$-tightness of $\{X^n\}_{n\in\NN}$ which means the time-continuity of the paths of the limit(s), whereas this is still under the Skorokhod topology. However, in view of the equivalence of $\mc C$-tightness in \cite[Proposition 3.26 (i) \& (ii), Page 351]{jacod2002limit}, the $\mc C$-tightness implies further that for all $\epsilon>0$ and $\eta>0$, there are $N_0\in\NN$ and $\theta\in(0,T]$ such that 
\begin{align}
\sup_{n\geq N_0} \PP\left( \sup_{r,s\in[0,T],\,|r-s|\leq \theta} |X^n_s-X^n_r|  >\eta \right)\leq \epsilon. \label{B-condition2}
\end{align}
Combining \eqref{B-condition1} and \eqref{B-condition2}, we may use the tightness criteria for probability measures on $\left(\mc{C}([0,T];\RR^d),\,\|\cdot\|_0\right)$ in \cite[Theorem 7.3, Page 82]{billingsley1999convergence} to obtain the tightness of $\{\mu_{X^n}\}_{n\in \NN}$ 
 on $\left(\mc{C}([0,T];\RR^d),\,\|\cdot\|_0\right)$ that is endowed with the uniform norm.  
 \end{proof}

\begin{thm}[Tightness]\label{thmtight}
Let Assumption \ref{asum} hold and $(X_t^m,V_t^m)_{t\in[0,T]}$ satisfy the system \eqref{MVeq}. For each countable subsequence $\{m_k\}_{k\in\mathbb N} \subset (0,1]$ with $\lim_{k\rightarrow \infty}m_k=0$, the sequence of probability distributions $\{\rho^{m_k}\}_{k\in\mathbb N}$ of $\{\OX^{m_k}\}_{k\in\mathbb N}$  is tight on $\left(\mc{C}([0,T];\RR^d),\,\|\cdot\|_0\right)$.
\end{thm}
\begin{proof}
It is sufficient to justify conditions $(Con 1)$ and $(Con 2)$ in Lemma \ref{lemAldous}.
	
	$\bullet$ \textit{Step 1: Checking $(Con 1)$. }  
	First, for $0< m\leq \frac{1}{2}$, recalling \eqref{onlyX}, we have by Fubini's theorem (see \cite[Theorem 4.33]{da2014stochastic} for the stochastic version)
		\begin{align}\label{Fubini}
	\OX_t^m&=\OX_0+\int_0^t  e^{-\frac{\gamma}{m}\tau}\OV_0d\tau
	+\frac{\lambda}{m}\int_0^t\int_0^\tau e^{-\frac{\gamma}{m}(\tau-s)}(X_s^{\alpha}(\rho^m)-\OX_s^m)dsd\tau\notag\\
	& \quad +\frac{\sigma}{m}\int_0^t\int_0^\tau e^{-\frac{\gamma}{m}(\tau-s)}D(X_s^{\alpha}(\rho^m)-\OX_s^m)dB_s d\tau \notag\\
	&=\OX_0+\int_0^t  e^{-\frac{\gamma}{m}\tau}\OV_0d\tau
	+\frac{\lambda}{m}\int_0^t\int_s^t e^{-\frac{\gamma}{m}(\tau-s)}d\tau (X_s^{\alpha}(\rho^m)-\OX_s^m)ds\notag\\
	& \quad +\frac{\sigma}{m}\int_0^t\int_s^t e^{-\frac{\gamma}{m}(\tau-s)}d\tau D(X_s^{\alpha}(\rho^m)-\OX_s^m)dB_s \notag\\
	&=\OX_0+\frac{m}{\gamma}(1-e^{-\frac{\gamma}{m}t})\OV_0
	+\frac{\lambda}{\gamma}\int_0^t(1-e^{-\frac{\gamma}{m}(t-s)}) (X_s^{\alpha}(\rho^m)-\OX_s^m)ds\notag\\
	& \quad +\frac{\sigma}{\gamma}\int_0^t(1-e^{-\frac{\gamma}{m}(t-s)}) D(X_s^{\alpha}(\rho^m)-\OX_s^m)dB_s\,.
	\end{align}
Here, the time $t\in[0,T]$ is deterministic and the stochastic Fubini's theorem is applicable as there holds the following integrability:
\begin{align*}
&\int_0^t\left(\EE \left[\int_0^\tau e^{-\frac{2\gamma}{m}(\tau-s)}|X_s^{\alpha}(\rho^m)-\OX_s^m|^2ds \right] \right)^\frac{1}{2}d\tau 
\\
&\leq  
\int_0^t\left(\int_0^\tau e^{-\frac{2\gamma}{m}(\tau-s)}C\EE[|\OX_s^m|^2]ds\right)^\frac{1}{2}d\tau 
	\leq (C\sup_{r\in[0,T]}\EE[|\OX_r^m|^2] )^{\frac{1}{2}}\int_0^t\left(\int_0^\tau e^{-\frac{2\gamma}{m}(\tau-s)}ds\right)^\frac{1}{2}d\tau
	\\
&
	\leq (C\sup_{r\in[0,T]}\EE[|\OX_r^m|^2] )^{\frac{1}{2}}T^\frac{1}{2}\left(\int_0^t\int_0^\tau e^{-\frac{2\gamma}{m}(\tau-s)}dsd\tau\right)^\frac{1}{2}
		=
		(C\sup_{r\in[0,T]}\EE[|\OX_r^m|^2] )^{\frac{1}{2}}T^\frac{1}{2}\left(\int_0^t\frac{m}{2\gamma}(1-e^{-\frac{2\gamma}{m}\tau})d\tau\right)^\frac{1}{2}
		\\
&
	\leq  (C\sup_{r\in[0,T]}\EE[|\OX_r^m|^2] )^{\frac{1}{2}}T(\frac{m}{2\gamma})^{\frac{1}{2}}<\infty\,,
\end{align*}
where we have used the fact that for all real $p\geq 1$,
	\begin{align}\label{Jensen}
	\EE[|X_t^{\alpha}(\rho^m)-\OX_t^m|^p]
	&=\int_{\RR^d}
	\left|\frac{\int_{\RR^d}x\omega_{\alpha}^{\mc{E}}(x)\rho^m(t,dx)}{\int_{\RR^d}\omega_{\alpha}^{\mc{E}}(x)\rho^m(t,dx)}-y\right|^p\rho^m(t,dy)
	=\int_{\RR^d}
	\left|\frac{\int_{\RR^d}(x-y)\omega_{\alpha}^{\mc{E}}(x)\rho^m(t,dx)}{\int_{\RR^d}\omega_{\alpha}^{\mc{E}}(x)\rho^m(t,dx)}\right|^p\rho^m(t,dy)\notag\\
	&\leq \frac{\int_{\RR^d}\int_{\RR^d}|x-y|^p\omega_{\alpha}^{\mc{E}}(x)\rho^m(t,dx)\rho^m(t,dy)}{\int_{\RR^d}\omega_{\alpha}^{\mc{E}}(x)\rho^m(t,dx)}\leq 2^pC_{\alpha,\TE}\EE[|\OX_t^m|^p],\quad \forall\, t\in[0,T].
	\end{align}
Please note that in the stochastic integral 
$$\frac{\sigma}{\gamma}\int_0^t(1-e^{-\frac{\gamma}{m}(t-s)}) D(X_s^{\alpha}(\rho^m)-\overline X_s^m)dB_s,$$ 
the integrand is not just a function of $s$ but also a function of $t$. Thus, the above stochastic integral is typically of the Volterra type and as a stochastic process of time $t$, it is not a local martingale except for certain trivial cases, so the  BDG and Doob's inequalities for local martingales are not applicable herein.
In what follows, we use the fact that for any sequence $\{a_i\}_{i=1}^n$ and $p\geq 1$, there holds
	\begin{equation*}
	\left(\Big| \sum_{i=1}^{n}a_i\Big|\right)^p\leq n^{p-1}\sum_{i=1}^{n}|a_i|^p\,.
	\end{equation*}
	Note that the assumption on $0< m\leq \frac{1}{2}$ ensures that $\gamma=1-m\in[\frac{1}{2},1)$, so $\frac{1}{\gamma}$ is well defined.
	It follows from H\"{o}lder's inequality that
{\small 	\begin{equation}\label{Xes}
	|\OX_t^m|^4\leq 64|\OX_0|^4+\frac{64m^4}{\gamma^4}|\OV_0|^4+\frac{64\lambda^4t^3}{\gamma^4}\int_0^t|X_s^{\alpha}(\rho^m)-\OX_s^m|^4ds+\frac{64\sigma^4}{\gamma^4}\left|\int_0^{t}(1-e^{-\frac{\gamma}{m}(t-s)}) D(X_s^{\alpha}(\rho^m)-\OX_s^m)dB_s\right|^4.
	\end{equation}}
Using the moment inequality for stochastic integrals as in \cite[Theorem 7.1]{mao2007stochastic}, we have
	\begin{align}\label{momentes}
	&\EE\left[\left|\int_0^{t}(1-e^{-\frac{\gamma}{m}(t-s)}) D(X_s^{\alpha}(\rho^m)-\OX_s^m)dB_s\right|^4\right]\nn\\
 &\leq d^3\EE\left[\sum_{k=1}^{d}\left|\int_0^{t}(1-e^{-\frac{\gamma}{m}(t-s)}) (X_s^{\alpha}(\rho^m)-\OX_s^m)_kdB_s^ke_k\right|^4\right]\nn\\
	&\leq 36d^3t\int_0^{t}\EE\left[  \sum_{k=1}^{d}|(X_s^{\alpha}(\rho^m)-\OX_s^m)_k|^4\right] ds\leq 36d^3t\int_0^{t}\EE\left[  |X_s^{\alpha}(\rho^m)-\OX_s^m|^4\right] ds\,.
	\end{align}
Thus,
	\begin{align*}
	\EE[|\OX_t^m|^4]\leq 64\EE[|\OX_0|^4]+\frac{64m^4}{\gamma^4}\EE[|\OV_0|^4]+\frac{64(\lambda^4t^3+36d^3t\sigma^4)}{\gamma^4}\int_0^t\EE[|X_s^{\alpha}(\rho^m)-\OX_s^m|^4]ds\,.
	\end{align*}
	Thus, we have
\begin{align*}
\EE[|\OX_t^m|^4]
\leq 64\EE[|\OX_0|^4]+\frac{64m^4}{\gamma^4}\EE[|\OV_0|^4]+\frac{1024C_{\alpha,\TE}(\lambda^4t^3+36d^3t\sigma^4)}{\gamma^4}\int_0^t\EE[|\OX_s^m|^4]ds\,.
\end{align*}
Using Gronwall's inequality leads to
\begin{equation}\label{esOX2}
\EE[|\OX_t^m|^4]
\leq \left(64\EE[|\OX_0|^4]+\frac{64m^4}{\gamma^4}\EE[|\OV_0|^4]\right)\exp\left(\frac{1024C_{\alpha,\TE}(\lambda^4T^3+36d^3T\sigma^4)}{\gamma^4}T\right),\quad\forall\, t\in[0,T]\,.
\end{equation}
Recalling  $0\leq m\leq \frac{1}{2}$ and $\frac{1}{\gamma}=\frac{1}{1-m}\leq 2$, from estimate \eqref{esOX2} we obtain the boundedness:
\begin{equation}\label{Sm}
\sup_{t\in[0,T]}\EE[|\OX_t^m|^4] \leq C(\EE[|\OX_0|^4],\EE[|\OV_0|^4],C_{\alpha,\TE},\lambda,d,\sigma,T)\,.
\end{equation}

Next, we consider the case when $\frac{1}{2}\leq m\leq 1$.
	It is obvious that
\begin{align}
|\OX_t^m|^4=|\OX_0|^4+4\int_0^t|\OX_s^m|^2\OX_s^m\cdot \OV_s^mds
&
\leq |\OX_0^m|^4+2\int_0^t|\OX_s^m|^2(|\OX_s^m|^2+|\OV_s^m|^2)ds
\nonumber\\
&
\leq |\OX_0^m|^4+3\int_0^t(|\OX_s^m|^4+|\OV_s^m|^4)ds\,. \label{esX}
\end{align}
Using \eqref{eqV}, arguments similar to \eqref{Xes}-\eqref{momentes} give
{\small \begin{align}\label{esV}
|\OV_t^m|^4
&\leq 64|\OV_0|^4+\frac{64\gamma^4}{m^4}\left|\int_0^t\OV_s^mds\right|^4+\frac{64\lambda^4}{m^4}\left|\int_0^t(X_s^{\alpha}(\rho^m)-\OX_s^m)ds\right|^4+\frac{64\sigma^4}{m^4}\left|\int_0^t D(X_s^{\alpha}(\rho^m)-\OX_s^m)dB_s\right|^4
\notag\\
&\leq 64|\OV_0|^4+\frac{64\gamma^4t^3}{m^4}\int_0^t|\OV_s^m|^4ds+\frac{64\lambda^4t^3}{m^4}\int_0^t|X_s^{\alpha}(\rho^m)-\OX_s^m|^4ds+\frac{64\sigma^4\cdot 36 d^3 t}{m^4}\int_0^t |X_s^{\alpha}(\rho^m)-\OX_s^m|^4ds\,.
\end{align}}
Collecting estimates \eqref{esV} and \eqref{esX}, and recalling the fact $\frac{1}{2}\leq m \leq 1$ and $0\leq \gamma\leq 1$, we have
\begin{align}
&\EE[|\OX_t^m|^4+|\OV_t^m|^4] \notag\\
\leq &64\, \EE[|\OX_0|^4+|\OV_0|^4]+ C \int_0^t\EE[|\OX_s^m|^4+|\OV_s^m|^4]ds+C 
\int_0^t\EE[|X_s^{\alpha}(\rho^m)-\OX_s^m|^4]ds\notag\\
\leq & 64\,\EE[|\OX_0|^4+|\OV_0|^4]
+ C 
(1+8C_{\alpha,\TE})\int_0^t\EE[|\OX_s^m|^4+|\OV_s^m|^4]ds
\,,
\end{align}
where the estimate \eqref{Jensen} is used in the last inequality. Applying Gronwall's inequality yields that 
\begin{align}\label{Lm}
&\EE[|\OX_t^m|^4+|\OV_t^m|^4] \leq 64\, \EE[|\OX_0|^4+|\OV_0|^4]\exp\left(C \cdot (1+8C_{\alpha,\TE})t\right),
\quad \forall\, t\in[0,T]\,.
\end{align}

Finally, combining \eqref{Sm} and \eqref{Lm} yields that
\begin{equation} \label{uniform-bd}
\sup_{m\in(0,1]} \sup_{t\in[0,T]}\EE[|\OX_t^m|^4] \leq C(\EE[|\OX_0|^4],\EE[|\OV_0|^4],C_{\alpha,\TE},\lambda,\sigma,d,T)=:C_1
\end{equation}
where the constant $C_1>0$ is independent of $m$. Therefore, for any $\varepsilon>0$, there exists a compact subset $K_\varepsilon:=\{x:~|x|^4\leq \frac{C_1}{\varepsilon}\}$ such that by Markov's inequality
\begin{equation}
\rho_t^m((K_\varepsilon)^c)=\PP(|X_t^m|^4> \frac{C_1}{\varepsilon})\leq \frac{\varepsilon\EE[|X_t^m|^4]}{C_1}\leq\varepsilon,\quad \forall ~m\in(0,1].
\end{equation}
This means that for each $t\in[0,T]$, each countable subset of $\{\rho_t^m\}_{0<m\leq 1}$ is tight, which verifies  condition $(Con 1)$ in Lemma \ref{lemAldous}.

$\bullet$ \textit{Step 2: Checking $(Con 2)$. }   Let $\beta$ be a $(\sigma(X^{m}_s;s\in[0,t]))_{t\geq 0}$-stopping time such that $\beta+\delta_0\leq T$. Without any loss of generality, we may assume that the concerned countable subsequence $\{m_k\}_{k\in\mathbb N} \subset [0,1]$ satisfies $m_k\leq \frac{1}{2}$ for all $k\in\NN$; thus, we may just consider the case of $0<m\leq \frac{1}{2}$ which indicates $\frac{1}{2}\leq \gamma<1$. 
Recall \eqref{onlyX} and compute
{\small 	\begin{align}\label{diff}
&\OX_{\beta+\delta}^m-\OX_{\beta}^m\notag\\
&=\int_\beta^{\beta+\delta}\OV_\tau d\tau=\int_\beta^{\beta+\delta}  e^{-\frac{\gamma}{m}\tau}\OV_0d\tau
+\frac{\lambda}{m}\int_\beta^{\beta+\delta}\int_0^\tau e^{-\frac{\gamma}{m}(\tau-s)}(X_s^{\alpha}(\rho^m)-\OX_s^m)dsd\tau
\notag\\
& \quad +\frac{\sigma}{m}\int_\beta^{\beta+\delta}e^{-\frac{\gamma\tau}{m}}\int_0^{\beta} e^{\frac{\gamma s}{m}}D(X_s^{\alpha}(\rho^m)-\OX_s^m)dB_s d\tau
+\frac{\sigma}{m} \int_\beta^{\beta+\delta} e^{-\frac{\gamma \tau}{m}} \int_{\beta}^{\tau}  e^{\frac{\gamma s}{m}}D(X_s^{\alpha}(\rho^m)-\OX_s^m)dB_s d\tau
 \notag\\
&=\int_\beta^{\beta+\delta}  e^{-\frac{\gamma}{m}\tau}\OV_0d\tau
+\frac{\lambda}{m}\int_0^\beta\int_\beta^{\beta+\delta} e^{-\frac{\gamma}{m}(\tau-s)}d\tau (X_s^{\alpha}(\rho^m)-\OX_s^m)ds+\frac{\lambda}{m}\int_\beta^{\beta+\delta}\int_s^{\beta+\delta} e^{-\frac{\gamma}{m}(\tau-s)}d\tau (X_s^{\alpha}(\rho^m)-\OX_s^m)ds
\notag\\
& \quad 
+\frac{\sigma}{m}\int_\beta^{\beta+\delta}e^{-\frac{\gamma \tau}{m}} d\tau \int_0^{\beta} e^{\frac{\gamma s}{m}}D(X_s^{\alpha}(\rho^m)-\OX_s^m)dB_s
+
\frac{\sigma}{m}\int_\beta^{\beta+\delta}\int_s^{\beta+\delta} e^{-\frac{\gamma}{m}(\tau-s)}d\tau D(X_s^{\alpha}(\rho^m)-\OX_s^m)dB_s \notag\\
&=\frac{m}{\gamma}(e^{-\frac{\gamma}{m}\beta}-e^{-\frac{\gamma}{m}(\beta+\delta)})\OV_0 \notag\\
&\quad +\frac{\lambda}{\gamma}\int_0^\beta (e^{-\frac{\gamma}{m}(\beta-s)}-e^{-\frac{\gamma}{m}(\beta+\delta-s)}) (X_s^{\alpha}(\rho^m)-\OX_s^m)ds
+\frac{\lambda}{\gamma}\int_\beta^{\beta+\delta} (1-e^{-\frac{\gamma}{m}(\beta+\delta-s)}) (X_s^{\alpha}(\rho^m)-\OX_s^m)ds\notag\\
& \quad 
+\frac{\sigma}{\gamma} ( e^{-\frac{\gamma \beta}{m}} -e^{-\frac{\gamma (\beta+\delta)}{m}} ) 
 \int_0^{\beta} e^{\frac{\gamma s}{m}}D(X_s^{\alpha}(\rho^m)-\OX_s^m)dB_s
+\frac{\sigma}{\gamma}\int_\beta^{\beta+\delta} (1-e^{-\frac{\gamma}{m}(\beta+\delta-s)})D(X_s^{\alpha}(\rho^m)-\OX_s^m)dB_s,
\end{align}}
where the calculations have taken into account the fact that $\beta$ is a stopping time.

Note that there holds $|e^{-x}-e^{-y}|\leq |x-y|^{\zeta}\wedge 1 $ for all $x,y\in[0,\infty)$ and $\zeta\in[0,1]$. Basic computations further indicate that for each $q\geq 1$, $\zeta\in[0,1]$, and $\tau\in[0,T]$,
\begin{align}
\int_0^{\tau}
\left| e^{-\frac{\gamma(\tau-s)}{m}} - e^{-\frac{\gamma(\tau+\delta-s)}{m}}   \right|^q \,ds
\leq
\int_0^{\tau}
\left( e^{-\frac{\gamma(\tau-s)}{m}} - e^{-\frac{\gamma(\tau+\delta-s)}{m}}  \right) \,ds
&
=
\frac{m}{\gamma} \left( 1- e^{-\frac{\gamma \delta}{m}}  \right)
-\frac{m}{\gamma} \left( e^{-\frac{\gamma \tau}{m}} -e^{-\frac{\gamma (\tau+\delta)}{m}} \right)
\nonumber\\
&\leq  \frac{m}{\gamma} \cdot \left(  \frac{\gamma\delta}{m}  \right)^{\zeta}
=\left(\frac{m }{\gamma} \right)^{1-\zeta}\delta^{\zeta},  \label{est-11}
\end{align}
and obviously, 
$$
\int_{\beta}^{\beta+\delta} \left(1-e^{-\frac{\gamma(\beta+\delta-s)}{m}}\right)^q ds \leq \int_{\beta}^{\beta+\delta} 1 \,ds =\delta.
$$
Then, it is obvious that
{\small \begin{align*}
\EE\left[\left|\frac{m}{\gamma}(e^{-\frac{\gamma}{m}\beta}-e^{-\frac{\gamma}{m}(\beta+\delta)})\OV_0\right|^2\right]
\leq \frac{m^2}{\gamma^2}\cdot \frac{\gamma^2\delta^2}{m^2}
\left(\EE[|\OV_0|^4]\right)^{\frac{1}{2}}
\leq \delta^2\left(\EE[|\OV_0|^4]\right)^{\frac{1}{2}}.
\end{align*}}
Next, it follows that
{\small \begin{align*}
&\EE\left[\left| \int_0^\beta (e^{-\frac{\gamma}{m}(\beta-s)}-e^{-\frac{\gamma}{m}(\beta+\delta-s)}) (X_s^{\alpha}(\rho^m)-\OX_s^m)ds\right|^2\right]
\notag\\
&\leq  
 \EE\left[\int_0^{\beta} |e^{-\frac{\gamma}{m}(\beta-s)}-e^{-\frac{\gamma}{m}(\beta+\delta-s)}|^2 ds\cdot \int_0^{\beta} |X_s^{\alpha}(\rho^m)-\OX_s^m|^2ds\right]\notag\\
&\leq
 \delta \cdot T \sup_{s\in[0,T]}  \left(\EE\left[    |X_s^{\alpha}(\rho^m)-\OX_s^m|^4 \right] \right)^{1/2},
\notag
\end{align*}}
and analogously,
{\small \begin{align*}
\EE\left[\left| \int_\beta^{\beta+\delta} (1-e^{-\frac{\gamma}{m}(\beta+\delta-s)}) (X_s^{\alpha}(\rho^m)-\OX_s^m)ds \right|^2\right] 
&\leq
\EE\left[
\int_{\beta}^{\beta+\delta} \left(1-e^{-\frac{\gamma(\beta+\delta-s)}{m}}\right)^2 ds
\cdot \int_\beta^{\beta+\delta} |X_s^{\alpha}(\rho^m)-\OX_s^m|^2ds
\right]
\\
&\leq   \delta\cdot
\EE\left[\int_\beta^{\beta+\delta} |X_s^{\alpha}(\rho^m)-\OX_s^m|^2ds\right] \\
& \leq  \delta \cdot T \sup_{s\in[0,T]}  \left(\EE\left[    |X_s^{\alpha}(\rho^m)-\OX_s^m|^4 \right]\right)^{1/2},\,.
\end{align*}}
Further, applying It\^{o}'s isometry  gives
\begin{align}
&\EE\left[\left|\int_{\beta}^{\beta+\delta} (1-e^{-\frac{\gamma}{m}(\beta+\delta-s)})D(X_s^{\alpha}(\rho^m)-\OX_s^m)d B_s\right|^2\right]
\notag\\
&=   
\EE\left[\int_{\beta}^{\beta+\delta} |1-e^{-\frac{\gamma}{m}(\beta+\delta-s)}|^2|X_s^{\alpha}(\rho^m)-\OX_s^m|^2ds\right]
\notag\\
& \leq \left(\EE\left[\int_{\beta}^{\beta+\delta} |1-e^{-\frac{\gamma}{m}(\beta+\delta-s)}|^4 ds \right] \right)^{1/2}
\cdot \left(\EE\left[\int_{0}^{T} |X_s^{\alpha}(\rho^m)-\OX_s^m|^4ds\right]\right)^{1/2}\notag\\
&\leq  (\delta T)^{1/2} \left(\sup_{s\in[0,T]}  \EE\left[    |X_s^{\alpha}(\rho^m)-\OX_s^m|^4 \right]\right)^{1/2}. \nonumber
\end{align}

Particularly, let us look at 
$$
Z^{m,\delta}_{t}:
= ( e^{-\frac{\gamma t}{m}} -e^{-\frac{\gamma (t+\delta)}{m}} ) 
 \int_0^{t} e^{\frac{\gamma s}{m}}D(X_s^{\alpha}(\rho^m)-\OX_s^m)dB_s, \quad t\in [0,T],
$$
and try to derive an estimate on $Z^{m,\delta}_{\beta}$. Here, we note that the main difficulty in estimating $Z^{m,\delta}_{\beta}$ arises from the fact that for a general stopping time $\beta$ (in fact, unless $\beta$ is $\sigma(\OX_0,\OV_0)$-measurable in our setting), the multiplier $( e^{-\frac{\gamma \beta}{m}} -e^{-\frac{\gamma (\beta+\delta)}{m}} )$ cannot enter the stochastic integral due to the nonanticipativity required for It\^o integrals and associated moment estimates.    Basic calculations as above yield that
\begin{align}
\EE\left[
\int_0^T \Big|  Z^{m,\delta}_{t} \Big|^4 dt
 \right]
&
=
\EE\left[
\int_0^T \Big|   
 \int_0^{t}  
 ( e^{-\frac{\gamma (t-s)}{m}} -e^{-\frac{\gamma (t+\delta-s)}{m}} )
 D(X_s^{\alpha}(\rho^m)-\OX_s^m)dB_s \Big|^4 dt
 \right]
 \nonumber\\
 &\leq C
\int_0^T \EE\left[ \Big(   
 \int_0^{t}  
 \Big|( e^{-\frac{\gamma (t-s)}{m}} -e^{-\frac{\gamma (t+\delta-s)}{m}} )
 (X_s^{\alpha}(\rho^m)-\OX_s^m)\Big|^2 ds \Big)^2 
 \right]\, dt
 \nonumber\\
  &\leq C
\int_0^T 
 \int_0^t \Big( e^{-\frac{\gamma (t-s)}{m}} -e^{-\frac{\gamma (t+\delta-s)}{m}} \Big)^4ds\cdot
 \EE\left[   
 \int_0^{t}  \Big|X_s^{\alpha}(\rho^m)-\OX_s^m\Big|^4 ds  
 \right]\, dt
 \nonumber\\
  &\leq C \, \delta  
\int_0^T 
  \EE\left[   
 \int_0^{t}  \Big|X_s^{\alpha}(\rho^m)-\OX_s^m\Big|^4 ds  
 \right]\, dt
 \nonumber\\
 &\leq
C \, \delta 
 \sup_{s\in[0,T]}  \EE\left[   |X_s^{\alpha}(\rho^m)-\OX_s^m|^4 \right], \label{est-111}
\end{align}
where we have used the so-called Burkholder-Davis-Gundy (BDG) inequality for martingales and the estimate \eqref{est-11} with $\zeta=1$ and the constant $C$ is independent of $m$ and $\delta$. Thus, the process
$$
M^{m,\delta}_t:= \int_0^t  (Z^{m,\delta}_s)' D(X_s^{\alpha}(\rho^m)-\OX_s^m)dB_s, \quad t\in[0,T],
$$
is a square-integrable continuous martingale; indeed, Doob's martingale inequality gives
\begin{align}
\EE\left[ \max_{t\in[0,T]} \left| M^{m,\delta}_t \right|^2 \right]
&\leq C\EE\left[ \int_0^T  \left| (Z^{m,\delta}_s)'  D(X_s^{\alpha}(\rho^m)-\OX_s^m) \right|^2 ds \right]
\nonumber\\
&\leq 
C \left(
\EE\left[ \int_0^T  \left|  Z^{m,\delta}_s \right|^4 ds \right] 
\EE\left[ \int_0^T  \left|  X_s^{\alpha}(\rho^m)-\OX_s^m \right|^4 ds \right] \right)^{1/2} \nonumber\\
&\leq
C  \delta^{1/2} 
 \sup_{s\in[0,T]}  \EE\left[    |X_s^{\alpha}(\rho^m)-\OX_s^m|^4 \right], \label{est-max}
\end{align}
with $C$ being independent of $m$ and $\delta$. On the other hand, it is easy to see that $Z^{m,\delta}$ satisfies the following SDE
$$
dZ^{m,\delta}_t= -\frac{\gamma}{m} Z^{m,\delta}_tdt + (1-e^{-\frac{\gamma \delta}{m}}) D(X_t^{\alpha}(\rho^m)-\OX_t^m)\,dB_t, \quad t>0;\quad Z^{m,\delta}_0=0.
$$
By It\^o-Doeblin formula, it holds that
for all $t\in[0,T]$,
\begin{align}
|Z^{m,\delta}_t|^2
&=\int_0^t \Big| ( e^{-\frac{\gamma (t-s)}{m}} -e^{-\frac{\gamma (t+\delta-s)}{m}} ) (X_s^{\alpha}(\rho^m)-\OX_s^m) \Big|^2ds 
+ 2 \int_0^t e^{-\frac{2\gamma (t-s)}{m}}(  1 -e^{-\frac{\gamma \delta}{m}}) d M^{m,\delta}_s
\nonumber\\
&\leq
\left( \int_0^t \Big| ( e^{-\frac{\gamma (t-s)}{m}} -e^{-\frac{\gamma (t+\delta-s)}{m}} \Big|^4ds   
\int_0^t \Big| X_s^{\alpha}(\rho^m)-\OX_s^m\Big|^4ds \right)^{1/2}
+2 (  1 -e^{-\frac{\gamma \delta}{m}}) M^{m,\delta}_t 
\nonumber\\
&\quad
	  - 4 \int_0^t \frac{\gamma}{m} e^{-\frac{2\gamma (t-s)}{m}}(  1 -e^{-\frac{\gamma \delta}{m}})  M^{m,\delta}_s \,ds
	\nonumber\\
&\leq
\left(   \delta 
\int_0^T \Big| X_s^{\alpha}(\rho^m)-\OX_s^m \Big|^4ds \right)^{1/2}
	+2\left| M^{m,\delta}_t \right|
	+ \frac{4\, \gamma}{m}  \max_{s\in[0,T]} \left| M^{m,\delta}_s \right| 
	 \int_0^t e^{-\frac{2\gamma (t-s)}{m}}  \,ds
\nonumber\\
&\leq
\left(   \delta 
\int_0^T \Big| X_s^{\alpha}(\rho^m)-\OX_s^m \Big|^4ds \right)^{1/2}
+4 \max_{s\in[0,T]} \left| M^{m,\delta}_s \right|  , \quad\text{a.s.,}\label{est-222}
\end{align}
where the integration by parts formula is applied to the stochastic integral in the first line and in the second inequality, we used estimate \eqref{est-11} with $\zeta$ equal to $ 1$. Combined with \eqref{est-max}, it yields that
\begin{align*}
\EE\left[ \Big|Z^{m,\delta}_{\beta} \Big|^2\right]
\leq  \EE\left[ \max_{t\in[0,T]}\Big|Z^{m,\delta}_{t} \Big|^2\right]
\leq C \left(\delta^{1/2}+ \delta^{1/4}\right) \left(  \sup_{s\in[0,T]}\EE\left[    |X_s^{\alpha}(\rho^m)-\OX_s^m|^4 \right]\right)^{1/2},
\end{align*}
where the constant $C$ is independent of $\beta,m,$ and $\delta$.

Therefore, summing up the above estimates and recalling  $0<m\leq \frac{1}{2}$, $\frac{1}{\gamma}\leq 2$, and the relations \eqref{Jensen} and \eqref{uniform-bd}, we arrive at
\begin{align}
\EE[|\OX_{\beta+\delta}^m-\OX_{\beta}^m|^2]  
& \leq {5}
\delta^2
(\EE[|\OV_0|^4])^{\frac{1}{2}} 
+\frac{5}{\gamma^2} \cdot 
\left(\lambda^2  \delta T + \sigma^2 C ( \delta^{1/2}+ \delta^{1/4} )   \right)
 \sup_{s\in[0,T]}  \left(\EE\left[    |X_s^{\alpha}(\rho^m)-\OX_s^m|^4 \right] \right)^{1/2}
\notag\\
& 
\leq
 C\left(\EE[|\OX_0|^4],\EE[|\OV_0|^4],C_{\alpha,\TE},\lambda,\sigma,d,T\right)
\left(\delta^{\frac{1}{4}} +\delta^{\frac{1}{2}}+\delta+\delta^2 \right). \nonumber
\end{align}
Hence, for any $\varepsilon>0$, $\eta>0$, there exists some $\delta_0>0$ such that for all $0<m\leq \frac{1}{2}$ it holds that
\begin{equation}
\sup_{\delta\in[0,\delta_0]}\PP(|\OX_{\beta+\delta}^m-\OX_{\beta}^m|^2\geq \eta)\leq \sup_{\delta\in[0,\delta_0]}\frac{\EE[|\OX_{\beta+\delta}^m-\OX_{\beta}^m|^2]}{\eta}\leq \varepsilon\,.
\end{equation}
This justifies condition $Con 2$ in Lemma \ref{lemAldous}.  
	\end{proof}


In a similar way to \eqref{Jensen}, it holds that for all real $p\geq 1$ and $\tau\in[0,T]$,
	\begin{align}
	\EE\left[\sup_{t\in[0,\tau]}|X_t^{\alpha}(\rho^m)-\OX_t^m|^p\right]
	&=\EE\left[\sup_{t\in[0,\tau]}
	\left|\frac{\int_{\RR^d}(x-\OX_t^m)\omega_{\alpha}^{\mc{E}}(x)\rho^m(t,dx)}{\int_{\RR^d}\omega_{\alpha}^{\mc{E}}(x)\rho^m(t,dx)}\right|^p \right]
	\nonumber\\
	&\leq 
	\EE\left[\sup_{t\in[0,\tau]}
	 \frac{\int_{\RR^d}|x-\OX_t^m|^p\omega_{\alpha}^{\mc{E}}(x)\rho^m(t,dx)}{\int_{\RR^d}\omega_{\alpha}^{\mc{E}}(x)\rho^m(t,dx)}  \right]
	\notag\\
	&\leq 2^{p-1} C_{\alpha,\TE}
	\EE\left[\sup_{t\in[0,\tau]}
	  \Big(\EE[|\OX_t^m|^p]+|\OX_t^m|^p\Big)  \right]
	  \nonumber\\
	&\leq 2^pC_{\alpha,\TE}\EE\left[\sup_{t\in[0,\tau]}|\OX_t^m|^p\right]. \label{Jensen-unifm}
	\end{align}
Furthermore, the obtained uniform boundedness \eqref{uniform-bd} may be strengthened as follows. 
\begin{cor}\label{cor-est}
Let Assumption \ref{asum} hold and $(X_t^m,V_t^m)_{t\in[0,T]}$ satisfy the system \eqref{MVeq} in Theorem \ref{thmtight}. We have the following uniform boundedness
\begin{equation} \label{uniform-bdness}
\sup_{m\in(0,1]} \EE\left[\max_{t\in[0,T]}  |\OX_t^m|^4 \right] \leq C(\EE[|\OX_0|^4],\EE[|\OV_0|^4 ],C_{\alpha,\TE},\lambda,\sigma,d,T) <\infty.
\end{equation}
\end{cor}
\begin{proof}
For the unique strong solution $(\OX^m,\OV^m)$ to SDE \eqref{MVeq}, applying the standard martingale inequalities (e.g., see \cite[Chapter III, Page 110]{ikeda1989stochastic}) to the stochastic integrals in SDE \eqref{MVeq}, we may use the estimate \eqref{Jensen} with standard arguments (see \cite[Section 5 of Chapter 2]{krylov_controlled} for instance) to conclude that the solution $(\OX^m,\OV^m)$ has continuous trajectories and satisfies
\begin{equation}\label{secmen}
 \EE\left[ \max_{t\in[0,T]} \Big( |\OX_t^m|^4+|\OV_t^m|^4\Big)\right]\leq C\,  \EE\left[ |\OX_0|^4+|\OV_0|^4\right]\,,
\end{equation}
where the constant $C$ depends only on $\lambda,m,\sigma,T$, and $C_{\alpha,\mc{E}}$. Also, we recall the uniform boundedness \eqref{uniform-bd}.

When $0<m \leq \frac{1}{2}$, we may start with \eqref{Xes} and consider the involved stochastic integral:
$$
\int_0^{t}(1-e^{-\frac{\gamma}{m}(t-s)}) D(X_s^{\alpha}(\rho^m)-\OX_s^m)dB_s
=:\int_0^{t} D(X_s^{\alpha}(\rho^m)-\OX_s^m)dB_s -\theta_t, \quad t\in[0,T].
$$
The Burkholder-Davis-Gundy (BDG) inequality gives that
\begin{align}
\EE\left[
\max_{\tau\in[0,t]} \left|
\int_0^{\tau} D(X_s^{\alpha}(\rho^m)-\OX_s^m)dB_s
\right|^4
\right]
&\leq
C \, \EE\left[
\Big( \int_0^{t} \left|
X_s^{\alpha}(\rho^m)-\OX_s^m 
\right|^2ds\Big)^2
\right]
\leq CT^2C_1 <\infty. \label{est-1-nexp}
\end{align}
On the other hand, it is obvious that $\theta_t=\int_0^{t} e^{-\frac{\gamma}{m}(t-s)} D(X_s^{\alpha}(\rho^m)-\OX_s^m)dB_s$ satisfies uniquely the SDE:
\begin{align}
d\theta_t= -\frac{\gamma}{m} \theta_tdt + D(X_t^{\alpha}(\rho^m)-\OX_t^m) dB_t, \quad t>0;\quad \theta_0=0.
\end{align}
Straightforward calculation implies that
\begin{align}
\EE\left[   \left|\theta_t\right|^4 \right]
&=
\EE\left[
\Big|   
 \int_0^{t}  
  e^{-\frac{\gamma (t-s)}{m}}  
 D(X_s^{\alpha}(\rho^m)-\OX_s^m)dB_s \Big|^4
 \right]
 \nonumber\\
 &\leq C
 \EE\left[ \Big(   
 \int_0^{t}  
 \Big|  e^{-\frac{\gamma (t-s)}{m}}  
 (X_s^{\alpha}(\rho^m)-\OX_s^m)\Big|^2 ds \Big)^2 
 \right]
 \nonumber\\
 \text{(by \eqref{Jensen})}\quad &\leq C
 \int_0^t   e^{-\frac{4\gamma (t-s)}{m}}ds  \cdot   
 \int_0^t \EE\left[\left| \OX_s^m\right|^4\right]\,ds
 \nonumber\\
 \text{(by \eqref{uniform-bd})}\quad
 &\leq C\,m, \quad\forall\, t\in[0,T]. \label{est-bd-theta-26}
\end{align}
Set
$$
\xi_t=\int_0^t|\theta_s|^2(\theta_s)' D(X_s^{\alpha}(\rho^m)-\OX_s^m)dB_s,\quad t\in[0,T]\,
$$
with $(\theta_s)'$ being the transpose of $\theta_s$.
Then for each $\epsilon\in (0,1)$,  we have
\begin{align}
8\,\EE\left[  \max_{s\in[0,t]} \left|\xi_s\right| \right]
&\leq C \EE\left[ \Big(\int_0^t  |\theta_s|^6 \left|X_s^{\alpha}(\rho^m)-\OX_s^m\right|^2ds\Big)^{1/2}  \right]
\nonumber\\
\text{(by \eqref{Jensen-unifm})}\quad &\leq \frac{C}{\sqrt{\epsilon}} \EE\left[\int_0^t |\theta_s|^4ds  \right] 
+\frac{1}{4} \EE\left[  \max_{s\in[0,t]} \left|\theta_s\right|^4 \right]
+\epsilon \EE\left[  \max_{s\in[0,t]} \left|\OX_s^m\right|^4 \right]
\nonumber\\
 \text{(by \eqref{est-bd-theta-26})}\quad
 &\leq
\frac{C}{\sqrt{\epsilon}}+\frac{1}{4} \EE\left[  \max_{s\in[0,t]} \left|\theta_s\right|^4 \right]
+ \epsilon \EE\left[  \max_{s\in[0,t]} \left|\OX_s^m\right|^4 \right],\quad\forall\,t\in[0,T]. \label{est-xi}
\end{align}
By It\^o-Doeblin formula, it holds that for all $t\in[0,T]$,
\begin{align}
 |\theta_t|^4
 &= 
 \int_0^t 2 e^{-\frac{4\gamma(t-s)}{m}} \Big(|\theta_s|^2 \left|X_s^{\alpha}(\rho^m)-\OX_s^m\right|^2
 +2\, |\theta_s' D(X_s^{\alpha}(\rho^m)-\OX_s^m)|^2
 \Big)ds + 4\int_0^t e^{-\frac{4\gamma(t-s)}{m}} d\xi_s
 \nonumber\\
 &\leq 
 \frac{ C}{\epsilon}
 \left(\int_0^t e^{-\frac{4\gamma(t-s)}{m}} ds\right)^{2}    \max_{s\in[0,t]} \left|\theta_s\right|^4  
+\frac{\epsilon}{16C_{\alpha,\TE}}  \max_{s\in[0,t]} \left|X_s^{\alpha}(\rho^m)-\OX_s^m\right|^4  
  + 4\xi_t -\frac{16\gamma}{m}
  \int_0^t e^{-\frac{4\gamma(t-s)}{m}} \xi_s \,ds
 \nonumber\\
 &\leq
\frac{\tilde C m^2}{\epsilon} 
  \cdot  \max_{s\in[0,t]} \left|\theta_s\right|^4  
+\frac{\epsilon}{16C_{\alpha,\TE}}  \max_{s\in[0,t]} \left|X_s^{\alpha}(\rho^m)-\OX_s^m\right|^4  
  + 8\max_{s\in[0,t]}|\xi_s|,\quad \text{a.s.,} \label{est-theta}
\end{align}
where $\tilde C>0$ is independent of $(t,m,\epsilon)$. Choosing $\tilde m=\frac{\epsilon^{1/2}}{2\tilde C^{1/2}}\wedge \frac{1}{2}$ so that $\frac{\tilde C m^2}{\epsilon}\leq \frac{1}{4}$ and combining \eqref{Jensen-unifm}, \eqref{est-xi} and \eqref{est-theta}, we have 
\begin{align}
\EE\left[  \max_{s\in[0,t]} \left|\theta_s\right|^4 \right] 
\leq \frac{C}{\sqrt{\epsilon}} + 4\epsilon \EE\left[ \max_{s\in[0,t]} \left|\OX_s^m\right|^4\right],  \quad\forall\, m\in(0,\tilde m],
\end{align}
which together with \eqref{est-1-nexp} and \eqref{uniform-bd} inserted into  \eqref{Xes} implies that
\begin{align}
\EE\left[\max_{s\in[0,t]} \big|\OX_s^m\big|^4\right]
\leq C+ \frac{C}{\sqrt{\epsilon}} + C \epsilon \EE\left[  \max_{s\in[0,t]} \left|\OX_s^m\right|^4\right],\quad\forall\,t\in[0,T],
\end{align}
where the constant $C>0$ is independent of $(t, m,\epsilon)$. Therefore, by choosing $\epsilon=\frac{1}{2\,C}$ and setting $\tilde m$ accordingly, we  obtain the desired uniform boundedness for $m\in(0,\tilde m]$. When $\tilde m<m<1$, the proof is standard (see \cite[Section 5 of Chapter 2]{krylov_controlled} for instance) and it is omitted.
\end{proof}

 Next we shall identify the limit process, before which we recall a lemma on the stability estimate of the nonlinear term $X^\alpha(\rho)$.
\begin{lem}{\cite[Lemma 3.2]{carrillo2018analytical}}\label{lemsta} 
	Assume that $\rho,\hat \rho\in\mc{P}_4(\RR^{d})$. Then the following stability estimate holds
	\begin{equation}\label{lemstaeq}
	|X^\alpha(\rho)-X^\alpha(\widehat \rho)|\leq CW_2(\rho,\widehat \rho)\,,
	\end{equation}
	where $W_2$ is the $2$-Wasserstein distance, and $C$ depends only on $\alpha,L$,  $\int_{\RR^d}|x|^4\rho(dx)$, and $\int_{\RR^d}|x|^4\hat\rho(dx)$.
\end{lem}

\begin{thm}[Zero-inertia limit]\label{thmzero-limit} 
	Let Assumption \ref{asum} hold and $(X_t^m,V_t^m)_{t\in[0,T]}$ satisfy the system \eqref{MVeq}. Then as $m\rightarrow 0^+$, the sequence of  stochastic processes $\{\OX^m\}_{0< m\leq 1}$  converge weakly to $\overline X$, which is the unique solution to the following SDE:
	\begin{align}\label{MVCBO1}
\OX_t=\OX_0
+\lambda\int_0^t(X_s^{\alpha}(\rho)-\OX_s)ds +\sigma\int_0^t D(X_s^{\alpha}(\rho)-\OX_s)dB_s\,.
\end{align}
Moreover,  we have the following convergences:
\begin{equation}\label{est-convegence-m}
\EE\left[\max_{t\in[0,T]}|\OX^{m}_t-\OX_t|^2\right]\leq C\, \sqrt{m} ,\quad
\max_{t\in[0,T]}\EE\left[ |\OX^{m}_t-\OX_t|^2\right] \leq C\, m,
\end{equation}
where the constant $C$ depends on $\EE[|\OX_0|^4],\EE[|\OV_0|^4],C_{\alpha,\TE},\lambda,\sigma,d$,  and $T$.
\end{thm}
\begin{rmk}
	By \eqref{est-convegence-m}, it follows from the definition of Wasserstein distance that
	\begin{equation}
	\sup_{t\in[0,T]}W_2^2(\rho^m_t,\rho_t)\leq  \max_{t\in[0,T]}\EE\left[|\OX^m_t-\OX_t|^2\right]\leq C\, m \,,
	\end{equation}
	which is consistent with the result obtained in \cite[Theorem 1.3]{choi2020quantified}, where the authors obtained a quantified overdamped limit  (with the same rate $m$) of the singular Vlasov-Poisson-Fokker-Planck system to the aggregation-diffusion equation. Besides, the obtained (strong) convergence of $\OX^m$ to $ \OX$ in the path space $\mc{C}([0,T];\RR^d)$ implies and is obviously stronger than the convergence of $\{\rho^m_t \}_{m>0}$ to $ \rho_t$ for each time $t\geq 0$. 
\end{rmk}

\begin{proof}
By Theorem \ref{thmtight},  each subsequence $\{\OX^{m_k}\}_{k\in\mathbb N}$ with $m_k\leq 1/2$  converging to $0$ as $k\rightarrow \infty$  admits a subsequence (denoted w.l.o.g. by itself) that converges weakly.
By Skorokhod's representation theorem (see \cite[Theorem 6.7 on page
70]{billingsley1999convergence}),   we may find a common probability space $(\Omega,\mc{F},\PP)$ on which the quadruple $\{(\OX_0,\OX^{m_k},\OX,B)\}_{k\in\mathbb N}$ converge to some $(\widehat X_0,\widehat X,\OX, B)$ as random variables valued in $\RR^d\times \mc{C}([0,T];\RR^{3d})$ almost surely. Here $B$ is an identical $d$-dimensional Wiener process on $(\Omega,\mc{F},\PP)$. In particular, we have 
\begin{align}
\mathbb P\left( \lim_{k\rightarrow \infty} \|\OX^{m_k}-\widehat X\|_0 =0 \right)=1,\label{converge-as}
\end{align}
where we recall the uniform norm $\|\OX^{m_k}-\widehat X\|_0=\max_{t\in[0,T]} |\OX^{m_k}_t-\widehat X_t| $. It is obvious that $\OX_0=\widehat X_0$ a.s. In the following, we shall verify that the limit $\widehat X$ is indeed the unique solution $\OX$ to SDE \eqref{MVCBO1}.

Recall that the unique strong solution to an SDE may be regarded as a function of the initial value and the driving Wiener process (see \cite[Chapter IV]{ikeda1989stochastic}). Thus, due to the uniqueness and existence of strong solution to SDE \eqref{Fubini} in Theorem \ref{thm-huang2021note},  on the above probability space $(\Omega,\mc{F},\PP)$ we must have
\begin{align}
\OX_t^{m_k}&=\OX_0+\frac{m_k}{\gamma}(1-e^{-\frac{\gamma}{{m_k}}t})\OV_0
+\frac{\lambda}{\gamma}\int_0^t(1-e^{-\frac{\gamma}{{m_k}}(t-s)}) (X_s^{\alpha}(\rho^{m_k})-\OX_s^{m_k})ds\notag\\
& \quad +\frac{\sigma}{\gamma}\int_0^t(1-e^{-\frac{\gamma}{{m_k}}(t-s)}) D(X_s^{\alpha}(\rho^{m_k})-\OX_s^{m_k})dB_s\,. \label{eq-X-mk}
\end{align}
By the estimates in Corollary \ref{cor-est} and Fatou's lemma,  
there exists a constant $C_2 $ such that
\begin{equation}\label{estimate-L4-bd}  
\sup_{k\in\mathbb N} \EE \left[\max_{t\in[0,T]}|\OX_t^{m_k}|^4\right]  +
  \EE\left[\max_{t\in[0,T]}\left|\widehat X_t\right|^4 \right] \leq  C_2:=C(\EE[|\OX_0|^4],\EE[|\OV_0|^4],C_{\alpha,\TE},\lambda,\sigma,d,T)<\infty.
\end{equation}
As a straightforward consequence of the above boundedness, it holds that
\begin{align}
\sup_{k\in \NN, t\in[0,T]} \PP(|\OX_t^{m_k} -\widehat X_t| > A)
\leq \frac{2^4C_2}{A^4},\quad \forall\, A>0.
\end{align}
Thus, the dominated convergence theorem gives that for each $A>0$,
\begin{align}
&\limsup_{k\rightarrow \infty}\EE\left[\int_0^T |\OX_t^{m_k}-\widehat X_t|^2\,dt \right] 
\nonumber\\
\leq&
\limsup_{k\rightarrow \infty}\left(\EE\left[\int_0^T |\OX_t^{m_k}-\widehat X_t|^2 \wedge A^2\,dt \right]
+ \EE \left[\int_0^T |\OX_t^{m_k}-\widehat X_t|^2 1_{\{|\OX_t^{m_k}-\widehat X_t|>A\}}\,dt \right]  \right)
\nonumber \\
\leq&
\limsup_{k\rightarrow \infty}\EE\left[\int_0^T |\OX_t^{m_k}-\widehat X_t|^2 \wedge A^2\,dt \right]
+  
T\cdot \sup_{k\in\mathbb N}   \left(\EE \left[\max_{t\in[0,T]}|\OX_t^{m_k}-\widehat X_t|^4\right] \right)^{1/2}   \left|  \PP(|\OX_t^{m_k} -\widehat X_t| > A) \right|^{1/2}
\nonumber
\\
\leq&
\limsup_{k\rightarrow \infty}\EE\left[\int_0^T |\OX_t^{m_k}-\widehat X_t|^2 \wedge A^2\,dt \right]
+  \frac{2^4 \, C_2 T}{A^2}
\nonumber
\\
=& \frac{2^4 \, C_2 T}{A^2}, \nonumber
\end{align}
which by the arbitrariness of $A>0$ indicates that 
\begin{equation} \label{limit-zero}
\lim_{k\rightarrow \infty}\EE\left[\int_0^T |\OX_t^{m_k}-\widehat X_t|^2\,dt \right] 
=0.
\end{equation}
Letting $\rho(t,dx)$  be the probability distribution of $\widehat X_t$ for $t\in[0,T]$, we have
$$
|{X}_{t}^{\alpha}(\rho)|=\left|\frac{\int_{\RR^d}x\omega_{\alpha}^{\mc{E}}(x)\rho(t,dx)}{\int_{\RR^d}\omega_{\alpha}^{\mc{E}}(x)\rho(t,dx)}\right| \leq C_{\alpha,\TE}\int_{\RR^d}|x|\rho(t,dx)\leq C_{\alpha,\TE}(\EE[|\widehat X_t|^4])^{\frac{1}{4}},
$$
and
\begin{equation}\label{Xalbound}
\sup_{k\in\mathbb N} \sup_{t\in[0,T]}|{X}_{t}^{\alpha}(\rho^{m_k})|\leq  C_{\alpha,\TE}(C_2)^{\frac{1}{4}}, \quad \text{ and }\quad
\sup_{t\in[0,T]} |{X}_{t}^{\alpha}(\rho)|\leq  C_{\alpha,\TE}(C_2)^{\frac{1}{4}}\,.
\end{equation}

Then we compute the limit of \eqref{eq-X-mk} term by term. By Lemma \ref{lemsta}, it holds that
\begin{equation*}
 |X_t^\alpha(\rho^{m_k})-X_t^\alpha( \rho)|^2\leq C W_2^2(\rho_t^{m_k}, \rho_t)\leq C \EE[|\OX_t^{m_k}-\widehat X_t|^2],
\end{equation*}
and thus, by recalling that $1\geq\gamma=1-m_k\geq \frac{1}{2}$, we have for each $t\in[0,T]$,
\begin{align}
&\EE \left[ \max_{\tau\in[0,t]}\left|
\frac{\lambda}{\gamma}
 \int_0^{\tau}(1-e^{-\frac{\gamma}{{m_k}}(t-s)}) (X_s^{\alpha}(\rho^{m_k})-\OX_s^{m_k})ds
- \lambda \int_0^{\tau}  (X_s^{\alpha}(\rho)-\widehat X_s)ds \right|^2\right]
\nonumber \\
& \leq 
2t \EE \left[  
\frac{\lambda}{1-m_k}
 \int_0^t\left|(1-e^{-\frac{1-m_k}{{m_k}}(t-s)}) (X_s^{\alpha}(\rho^{m_k})- X_s^{\alpha}(\rho) + \widehat X_s-\OX_s^{m_k}) \right|^2ds   \right]
 \nonumber \\
 &\quad
 	+2t \EE \left[  
 {\lambda} 
 \int_0^t\left| \left(\frac{1-e^{-\frac{1-m_k}{{m_k}}(t-s)}}{1-m_k} -1\right) (X_s^{\alpha}(\rho)- \widehat X_s)\right|^2ds  
\right]
\nonumber \\
&\leq
C  \EE \left[ \int_0^t  \left|   \widehat X_s-\OX_s^{m_k} \right|^2 ds \right]
+ C 
 {\lambda^2}  \int_0^t \left|\frac{1-e^{-\frac{1-m_k}{{m_k}}(t-s)}}{1-m_k} -1\right|^2ds 
 \cdot \EE \left[  \int_0^T \left|X_s^{\alpha}(\rho)- \widehat X_s\right|^2  ds 
\right] 
\nonumber \\
&\leq
C \EE \left[ \int_0^t  \left|   \widehat X_s-\OX_s^{m_k} \right|^2 ds \right]
+ C 
  \int_0^t \left|\frac{1-e^{-\frac{1-m_k}{{m_k}}(t-s)}-(1-m_k)}{1-m_k} \right|^2ds  
 \nonumber \\
 &\leq 
 C \EE \left[ \int_0^t  \left|   \widehat X_s-\OX_s^{m_k} \right|^2 ds \right]
+ C 
  \int_0^t \left( \left|{m_k}\right|^2  +  e^{-\frac{2(1-m_k)}{{m_k}}(t-s)} \right) ds 
  \nonumber \\
  & \leq
  C \EE \left[ \int_0^t  \left|   \widehat X_s-\OX_s^{m_k} \right|^2 ds \right]
+ C 
   \left( t \left|{m_k}\right|^2  + \frac{m_k}{2(1-m_k)}   \right)  , \label{est-drift}
\end{align}
where the constants $C$s are independent of $(k,t)$.

For the stochastic integrals, using Doob's inequality gives 
\begin{align}
&\EE \left[ \max_{\tau\in[0,t]}\left|
\frac{\sigma}{\gamma}
 \int_0^{\tau} D(X_s^{\alpha}(\rho^{m_k}_s)-\OX_s^{m_k})dB_s
- \sigma \int_0^{\tau}  D(X_s^{\alpha}(\rho)-\widehat X_s)dB_s \right|^2\right]
\nonumber \\
& \leq C
 \EE \left[ 
 \int_0^t \left| \frac{1 } {\gamma}   (X_s^{\alpha}(\rho^{m_k})-\OX_s^{m_k}) 
-  (X_s^{\alpha}(\rho)-\widehat X_s)\right|^2 ds \right]
\nonumber \\
&\leq
C
\EE \left[ 
 \int_0^t \frac{1} {\gamma^2} \left|    (X_s^{\alpha}(\rho^{m_k})-\OX_s^{m_k}) -     (X_s^{\alpha}(\rho)-\widehat X_s)\right|^2\,ds\right]
+  C\EE \left[ 
 \int_0^t \left|
	  \left(
	  \frac{1 } {\gamma} -1
	  \right)
	  (X_s^{\alpha}(\rho)-\widehat X_s)\right|^2 ds \right]
\nonumber \\
&\leq
C
\EE \left[ 
 \int_0^t \left|     \OX_s^{m_k}     -\widehat X_s  \right|^2\,ds\right]
 +C |m_k|^2
 , \quad \forall\,t\in[0,T]. \label{converge-m10}
\end{align}
Meanwhile, putting
$$
Z^k_t:=\frac{\sigma}{\gamma} \int_0^{t} e^{-\frac{\gamma}{{m_k}}(t-s)} D(X_s^{\alpha}(\rho^{m_k})-\OX_s^{m_k})dB_s,\quad t\in[0,T],
$$
we use It\^o's isometry and the relations \eqref{Jensen-unifm} and \eqref{estimate-L4-bd} to obtain
\begin{align}
\max_{t\in[0,T]}\EE \left[\left| Z^k_t \right|^2\right]
&=\max_{t\in[0,T]} \EE\left[
\frac{\sigma^2}{\gamma^2}   \int_0^{t} 
e^{-\frac{2 \gamma}{m_k}(t-s)} 
\left| X_s^{\alpha}(\rho^{m_k})-\OX_s^{m_k}\right|^2\,ds
\right]
\nonumber\\
&\leq \max_{t\in[0,T]}
\frac{\sigma^2}{\gamma^2} 
\left(\sup_{s\in[0,T]}\EE\left[  \left| X_s^{\alpha}(\rho^{m_k})-\OX_s^{m_k}\right|^4 \right]\right)^{1/2}
  \int_0^{t} 
e^{-\frac{2 \gamma}{m_k}(t-s)} ds 
\nonumber\\
&\leq C \, m_k,\label{est-243-Z}
\end{align}
where $C$ is independent of $(m_k,t)$. Then the process
\begin{align}
\xi_t:=\int_0^t (Z^k_s)' D(X_s^{\alpha}(\rho^{m_k})-\OX_s^{m_k})dB_s,\quad t\in[0,T]\,\label{def-xi-t}
\end{align}
with $(Z^k_s)'$ being the transpose of $Z^k_s$,
defines a continuous martingale with
\begin{align}
\EE\left[  \max_{s\in[0,T]} \left|\xi_s\right| \right]
&\leq C \EE\left[ \Big(\int_0^T  |Z^k_s|^2 \left|X_s^{\alpha}(\rho^{m_k})-\OX_s^{m_k}\right|^2ds\Big)^{1/2}  \right]
\nonumber\\
&\leq C \left( \EE\left[\int_0^T |Z^k_s|^2ds  \right] \right)^{1/2}
 \cdot \left(\EE\left[ \sup_{s\in[0,T]} \left|X_s^{\alpha}(\rho^{m_k})-\OX_s^{m_k}\right|^4 \right] \right)^{1/4}
\nonumber\\
\text{(by \eqref{Jensen-unifm}, \eqref{estimate-L4-bd}, and \eqref{est-243-Z})}\quad &\leq C \, |m_k|^{1/2}. \label{est-martingale-Z}
\end{align}
Using the It\^o-Doeblin formula as in \eqref{est-theta} gives that
\begin{align}
 |Z^k_t|^2
 &= 
 \int_0^t  e^{-\frac{2\gamma(t-s)}{m_k}}  \left|X_s^{\alpha}(\rho^{m_k})-\OX_s^{m_k}\right|^2
  ds 
  + 2\int_0^t e^{-\frac{2\gamma(t-s)}{{m_k}}} d\xi_s
  \label{Ito-Doeblin}\\
 &\leq  
 C \int_0^t e^{-\frac{2\gamma(t-s)}{m_k}} ds \cdot     \sup_{s\in[0,t]} \left|X_s^{\alpha}(\rho^{m_k})-\OX_s^{m_k}\right|^2  
  + 2\xi_t -\frac{4\gamma}{m_k}
  \int_0^t e^{-\frac{2\gamma(t-s)}{m_k}} \xi_s \,ds
 \nonumber\\
 &\leq
C{m_k}  \sup_{s\in[0,t]} \left| X_s^{\alpha}(\rho^{m_k})-\OX_s^{m_k}\right|^2  
  + 4\max_{s\in[0,t]}|\xi_s|,\quad \text{a.s.} \label{est-theta-Z}
\end{align}
Putting \eqref{Jensen-unifm}, \eqref{estimate-L4-bd}, \eqref{est-martingale-Z}, and \eqref{est-theta-Z} together, we obtain that
\begin{align}
&
\EE \left[ \max_{\tau\in[0,t]} \left|Z^k_{\tau}\right|^2    \right]
\leq  C\, \sqrt{m_k}.
\label{ineq-converge-Zmax}
\end{align}
In addition, it is obvious that
\begin{align}
 \left| \frac{m_k}{\gamma}(1-e^{-\frac{\gamma}{m_k}t})\OV_0\right|
\leq  C m_k \left| \OV_0\right|. \label{est-V0}
\end{align}

As the constants $C$s are independent of $(k,t)$, combining the estimates \eqref{est-drift}, \eqref{converge-m10}, \eqref{ineq-converge-Zmax} and \eqref{est-V0}, letting $k$ tend to infinity on both sides of \eqref{eq-X-mk} and recalling $\frac{1}{2}\geq m_k \rightarrow 0^+$ and the relation \eqref{limit-zero}, we have
	\begin{align*}
\widehat X_t=\OX_0
+\lambda\int_0^t(X_s^{\alpha}(\rho)-\widehat X_s)ds +\sigma\int_0^t D(X_s^{\alpha}(\rho)-\widehat X_s)dB_s.
\end{align*}
Therefore, the limit $\widehat X$ turns out to be a (strong) solution to SDE \eqref{MVCBO1}.
Meanwhile, in view of the continuity of $X^{\alpha}(\rho)$ in Lemma \ref{lemsta}, we can easily verify the uniqueness of the strong solution with standard arguments as in \cite[Theorem 3.1]{carrillo2018analytical}. Thus, we must have $\widehat X = \OX$ that is the unique strong solution to SDE \eqref{MVCBO1} with $ \EE\left[ \max_{t\in[0,T]}|\OX_t|^4\right]\leq C_2$. Further, due to the arbitrariness of the subsequence $\{\OX^{m_k}\}_{k\in\mathbb N}$, we conclude that as $m\rightarrow 0^+$, the sequence of  stochastic processes $\{\OX^m\}_{0< m\leq 1}$  converge weakly to  the unique solution $\overline X$ to SDE \eqref{MVCBO1}.  

Finally, to measure the distance between $\OX^m$ and the limit $\widehat X=\OX$, we may replace $\widehat X$ with $\OX$ in the calculations \eqref{est-drift}, \eqref{converge-m10}, \eqref{ineq-converge-Zmax}, and \eqref{est-V0} and arrive at
\begin{align}
\EE\left[\max_{\tau\in[0,t]}|\OX^{m}_{\tau}-\OX_{\tau}|^2\right]
\leq C\int_0^t\EE[|\OX^{m}_s-\OX_s|^2]ds+ 
C\, \sqrt{m},\quad t\in[0,T], \nonumber
\end{align}
where the constant $C$ is independent of $(m,t)$; if we adopt the estimate \eqref{est-243-Z} instead of \eqref{ineq-converge-Zmax} in the above, it holds that 
\begin{align}
\EE\left[ |\OX^{m}_{t}-\OX_{t}|^2\right]
\leq C\int_0^t\EE[|\OX^{m}_s-\OX_s|^2]ds+ 
C\, m,\quad t\in[0,T], \nonumber
\end{align}
with $C$ independent of $(m,t)$. By Gronwall's inequality, we may have
\begin{align}
\EE\left[\max_{t\in[0,T]}|\OX^{m}_t-\OX_t|^2\right]&\leq C\, \sqrt{m} ,\quad
\max_{t\in[0,T]}\EE\left[ |\OX^{m}_t-\OX_t|^2\right] \leq C\, m ,\label{convergence-l2}
\end{align}
where the constants $C$s depend on $\EE[|\OX_0|^4],\EE[|\OV_0|^4],C_{\alpha,\TE},\lambda,\sigma,d$, and $T$.   
\end{proof}

\begin{rmk}\label{remark_convergence}
When proving the convergence of   $\{\OX^m\}$ satisfying SDEs \eqref{MVeq}  to the solution  $\OX$ of \eqref{MVCBO}, we  cannot expect the convergence of the associated velocity processes $\{\OV^m\}$ due to the indifferentiability of the limit $\{\OX_t\}_{t\geq 0}$ with respect to time $t$ if $\sigma\neq 0$. Therefore, we do not investigate the convergence  of the joint Markovian process $\{(\OX^m,\OV^m)\}$ and consider instead solely the process $\{\OX^m\}$ which satisfies a stochastic equation \eqref{eq-X-mk} of Volterra type, being path-dependent and thus non-Markovian. This non-Markovianity prevents us from using the usual techniques for weak convergence with martingale problems but prompts us to identify the limit by measuring directly the distance between $\OX^{m_k}$ and $\widehat X=\OX$  in the above proof.

Regarding to the rate of convergence in \eqref{convergence-l2}, the $L^2$-estimate of the uniform norm of the paths of $\OX^{m}-\OX$  is on the order of $\sqrt{m}$ while the $L^2$-estimate for each time $t\in[0,T]$ has order of $m$. Such a difference stems from the relation \eqref{Ito-Doeblin} which also reads
$$
|Z^k_t|^2
 = 
 \int_0^t  e^{-\frac{2\gamma(t-s)}{m_k}}  \left|X_s^{\alpha}(\rho^{m_k})-\OX_s^{m_k}\right|^2
  ds 
  + 2\int_0^t e^{-\frac{2\gamma(t-s)}{{m_k}}} (Z^k_s)' D(X_s^{\alpha}(\rho^{m_k})-\OX_s^{m_k})dB_s.
$$
Here the stochastic integral has mean zero and taking expectation on both sides can directly give us the order of $m$ as seen in the first term on the right hand side  of \eqref{est-theta-Z}; however, when considering the uniform norm, we cannot ignore the stochastic integral and its \textit{linear} dependence on $(Z^k_s)_{s\in[0,T]}$ leads to the order $\sqrt{m}$. In addition, in the above proof,  because both $\OX^{m_k}$ and $\OX$ are the strong solutions to associated SDEs with pathwise uniqueness and a strong solution is a functional of the initial value and the driving Wiener process (see \cite[Chapter IV]{ikeda1989stochastic}),   the following convergence
 \begin{equation}\label{strong_convergence}
\mathbb P\left( \lim_{k\rightarrow \infty} \|\OX^{m_k}-\OX\|_0=0 \right)=1,
\end{equation}
 actually holds for strong solutions to SDEs \eqref{MVeq} and \eqref{MVCBO1} on any given probability space $(\Omega,\mc{F},\PP)$ equipped with augmented filtration generated by a $d$-dimensional Wiener process $B$. Indeed, some experiments conducted at the end of Section 4 (see Figures \ref{fig:convergence_no_memory}, \ref{fig:convergence_memory}, and \ref{fig:comparison}) numerically support the convergences established here.
\end{rmk}

\section{Generalization to the case with memory effects}\label{section_memory}
In \cite{grassi2020particle}, the authors considered a PSO model which involves the memory of the local best positions, and it is of the form
{\small 
\begin{align}
&d\OX_t^m= \OV_t^mdt, \label{mXeq}\\
&d \OY_{t}^{m}=\nu\left(\OX_{t}^{m}-\OY_{t}^{m}\right) S^{\beta}\left(\OX_{t}^{m}, \OY_{t}^{m}\right) d t, \label{mYeq}\\
&d\OV_{t}^{m}=-\frac{\gamma}{m} \OV_{t}^{m} d t+\frac{\lambda_{1}}{m}\left(\OY_{t}^{m}-\OX_{t}^{m}\right) d t+\frac{\lambda_{2}}{m}\left({Y}_{t}^{\alpha}(\overline \rho^m)-\OX_{t}^{m}\right) d t \notag\\
&\quad  \qquad  +\frac{\sigma_{1}}{m} D\left(\OY_{t}^{m}-\OX_{t}^{m}\right) d B_{t}^{1}+\frac{\sigma_{2} }{m}D\left({Y}_{t}^{\alpha}(\overline \rho^m)-\OX_{t}^{m}\right) d B_{t}^{2}  \label{mVeq}\,,
\end{align}}
where $B^1$ and $B^2$ are two mutually independent d-dimensional Wiener processes, and similarly to the previous section, we introduce the following regularization of the global best
position
\begin{equation}\label{global_best_regular}
{Y}_t^{\alpha}(\overline \rho^m)=\frac{\int_{\mathbb{R}^{d}} y \omega_{\alpha}(y) {\overline \rho}^m(t,dy) }{\int_{\mathbb{R}^{d}} \omega_{\alpha}(y) {\bar \rho}^m(t,dy) }, \quad {\overline \rho}^m(t,y)=\iint_{\mathbb{R}^{d} \times \mathbb{R}^{d}} f^m(t,dx, y,d v)\,.
\end{equation}
Here the equation \eqref{mYeq} of $\OY^m$ is  the time continuous approximation to the evolution of the local best position, and
$S^{\beta}$ with $\beta \gg1$ is the hyperbolic tangent $S^{\beta}(x,y)=1+\tanh(\beta(\TE(y)-\TE(x))$. The corresponding mean-field PSO equation is
\begin{align}\label{mPSOeq}
\partial_{t} f_t^m+v \cdot \nabla_{x} f_t^m&+\nabla_{y} \cdot(\nu(x-y) S^{\beta}(x, y) f_t^m)
=  \nabla_{v} \cdot\bigg(\frac{\gamma}{m} v f_t^m+\frac{\lambda_{1}}{m}(x-y) f_t^m+\frac{\lambda_{2}}{m}(x-{Y}_t^{\alpha}({\bar \rho}^m)) f_t^m
\notag\\
&
+\Big(\frac{\sigma_{2}^{2}}{2 m^{2}} D(x-{Y}_t^{\alpha}({\bar \rho}^m))^{2}+\frac{\sigma_{1}^{2}}{2 m^{2}} D(x-y)^{2}\Big) \nabla_{v} f_t^m\bigg)\,.
\end{align}

We want to prove that the zero-inertia limit $(m\to 0)$ leads to the following mean-field CBO dynamic
\begin{align*}
\begin{cases}
\OX_t=\OX_0
+\lambda_1\int_0^t (\OY_s-\OX_s)ds +\sigma_1\int_0^tD(\OY_s-\OX_s)dB_s^1+\lambda_2\int_0^t (Y_s^{\alpha}(\bar\rho)-\OX_s)ds +\sigma_2\int_0^t D(Y_s^{\alpha}(\bar\rho)-\OX_s)dB_s^2,
\\
\OY_{t}=\OY_0+\nu\int_0^t\left(\OX_{s}-\OY_{s}\right) S^{\beta}\left(\OX_{s}, \OY_{s}\right) d s\,,
\end{cases}
\end{align*}
and its corresponding  partial differential equation is
\begin{align}\label{mCBOeq}
&\partial_t\rho_t+\nabla_{x} \cdot\left(\lambda_{1}(y-x)+\lambda_{2}({Y}_t^{\alpha}({\bar\rho})-x)\rho_t\right) +\nabla_{y} \cdot\left(\nu(x-y) S^{\beta}(x, y)) \rho_t\right)\notag\\
=&\frac{1}{2} \sum_{j=1}^{d} \frac{\partial^{2}}{\partial x_{j}^{2}}\left(\rho_t(\sigma_{1}^{2}(x-y)_j^{2}+\sigma_{2}^{2}(x-{Y}_{t}^{\alpha}({ \bar \rho}))_j^{2})\right)\,,
\end{align}
where $\bar\rho(t,y)=\int_{\RR^d}\rho(t,dx,y)$.

Since the proof of the zero-inertia limit for the PSO dynamics with memory effects follows similar arguments as developed in the previous section and  no essential innovation is needed to be explained, we only  sketch the proof for the tightness.
\begin{thm}[Tightness]\label{thmtight1}
	Let Assumption \ref{asum} hold and $(\OX_t^m,\OY_t^m,\OV_t^m)_{t\in[0,T]}$ satisfy the system \eqref{mXeq}-\eqref{mVeq}.  For each countable subsequence $\{m_k\}_{k\in\mathbb N} \subset [0,1]$ with $\lim_{k\rightarrow \infty}m_k=0$, the sequence of probability distributions   $\{\rho^{m_k}\}_{k\in\mathbb N}$ of $\{(\OX^{m_k},\OY^{m_k})\}_{k\in\mathbb N}$  is tight.
\end{thm}
\begin{proof}
	The proof is similar to Theorem \ref{thmtight}.	
	
	$\bullet$ \textit{Step 1: Checking $(Con 1)$. }   For $0<m\leq \frac{1}{2}$, we first solve  \eqref{mVeq}  for $\OV^m$ and obtain
	\begin{align*}
	\OV_t^m&=e^{-\frac{\gamma}{m}t}\OV_0+\frac{\lambda_1}{m}\int_0^te^{-\frac{\gamma}{m}(t-s)}(\OY_s^m-\OX_s^m)ds+\frac{\sigma_1}{m}\int_0^te^{-\frac{\gamma}{m}(t-s)}D(\OY_s^m-\OX_s^m)dB_s^1\notag\\
	&\quad+\frac{\lambda_2}{m}\int_0^te^{-\frac{\gamma}{m}(t-s)}(Y_s^{\alpha}(\bar\rho^m)-\OX_s^m)ds+\frac{\sigma_2}{m}\int_0^te^{-\frac{\gamma}{m}(t-s)}D(Y_s^{\alpha}(\bar\rho^m)-\OX_s^m)dB_s^2\,.
	\end{align*}
	Here $\bar\rho^m(t,y)=\int_{ \RR^d }\rho^m(t,dx,y)$.
	By Fubini's theorem,  similar arguments as in \eqref{Fubini} yield that
	\begin{align}\label{Fubini1}
	\OX_t^m&=\OX_0+\frac{m}{\gamma}(1-e^{-\frac{\gamma}{m}t})\OV_0
	+\frac{\lambda_1}{\gamma}\int_0^t(1-e^{-\frac{\gamma}{m}(t-s)}) (\OY_s^m-\OX_s^m)ds +\frac{\sigma_1}{\gamma}\int_0^t(1-e^{-\frac{\gamma}{m}(t-s)}) D(\OY_s^m-\OX_s^m)dB_s^1\notag\\
	&\quad +\frac{\lambda_2}{\gamma}\int_0^t(1-e^{-\frac{\gamma}{m}(t-s)}) (Y_s^{\alpha}(\bar\rho^m)-\OX_s^m)ds +\frac{\sigma_2}{\gamma}\int_0^t(1-e^{-\frac{\gamma}{m}(t-s)}) D(Y_s^{\alpha}(\bar\rho^m)-\OX_s^m)dB_s^2
	\,.
	\end{align}
	Following the same computations as in Theorem \ref{thmtight} gives
	\begin{align*}
	\EE[|\OX_t^m|^4]
	\leq C \EE[|\OX_0|^4+|\OV_0|^4]+C\int_0^t\EE[|Y_s^{\alpha}(\rho^m)-\OX_s^m|^4]ds+C\int_0^t\EE[|\OY_s^m-\OX_s^m|^4]ds\,,
	\end{align*}
	where $C$ depends only on $\lambda_1,\sigma_2,\lambda_2,\sigma_2,d$, and $T$. 
	Put $\tilde\rho^m(t,x)=\int_{\RR^d}\rho^m(t,x,dy)$. In a similar way to \eqref{Jensen} we have
	\begin{align*}
	\EE[|Y_t^{\alpha}(\bar\rho^m)-\OX_t^m|^4]&=\int_{\RR^d}
	\left|\frac{\int_{\RR^d}y\omega_{\alpha}^{\mc{E}}(y)\bar\rho^m(t,dy)}{\int_{\RR^d}\omega_{\alpha}^{\mc{E}}(y)\bar\rho^m(t,dy)}-x\right|^4\tilde\rho^m(t,dx)=\int_{\RR^d}
	\left|\frac{\int_{\RR^d}(y-x)\omega_{\alpha}^{\mc{E}}(y)\bar\rho^m(t,dy)}{\int_{\RR^d}\omega_{\alpha}^{\mc{E}}(y)\bar\rho^m(t,dy)}\right|^4\tilde\rho^m(t,dx)\notag\\
	&\leq \frac{\int_{\RR^d}\int_{\RR^d}|x-y|^4\omega_{\alpha}^{\mc{E}}(y)\bar\rho^m(t,dy)\tilde\rho^m(t,dx)}{\int_{\RR^d}\omega_{\alpha}^{\mc{E}}(y)\bar\rho^m(t,dy)}\leq 8C_{\alpha,\TE}\EE[|\OX_t^m|^4+|\OY_t^m|^4]\,.
	\end{align*}
	Thus it yields that
	\begin{align}\label{esX1}
	\EE[|\OX_t^m|^4]
	\leq C \EE[|\OX_0|^4+|\OV_0|^4]+C\int_0^t\EE[|\OY_s^m|^4+|\OX_s^m|^4]ds\,,
	\end{align}
	where $C$ depends only on $\lambda_1,\sigma_2,\lambda_2,\sigma_2,d,T$ and $C_{\alpha,\TE}$.
	
	Recall that
	\begin{align*}
	\OY_{t}^{m}=\OY_0+\nu\int_0^t\left(\OX_{s}^{m}-\OY_{s}^{m}\right) S^{\beta}\left(\OX_{s}^{m}, \OY_{s}^{m}\right) ds
	\end{align*}
	with $S^{\beta}(x,y)=1+\tanh(\beta(\TE(y)-\TE(x))$. Using the fact that $|S^{\beta}|\leq 2$ then it follows
	\begin{equation*}
	\EE[|\OY_t^m|^4]
	\leq C\EE[|\OY_0|^4]+C\int_0^t\EE[|\OY_s^m|^4+|\OX_s^m|^4]ds\,,
	\end{equation*}
	where $C$ depends only on $\nu$ and $T$. This together with \eqref{esX1} implies
	\begin{equation*}
	\EE[|\OX_t^m|^4+|\OY_t^m|^4]
	\leq C \EE[|\OX_0|^4+|\OY_0|^4+|\OV_0|^4]+C\int_0^t\EE[|\OY_s^m|^4+|\OX_s^m|^4]ds\,.
	\end{equation*}
	By Gronwall's inequality, it yields that
	\begin{equation}\label{uniform}
	\sup_{t\in[0,T]}\EE[|\OX_t^m|^4+|\OY_t^m|^4]\leq C(\EE[|\OX_0|^4+|\OY_0|^4+|\OV_0|^4],\lambda_1,\sigma_2,\lambda_2,\sigma_2,d,T,C_{\alpha,\TE},\nu)\,,
	\end{equation}
	which verifies  $(Con 1)$ for the case of $0<m\leq \frac{1}{2}$. We omit the discussions for the case of $\frac{1}{2}< m\leq 1$.
	
	$\bullet$ \textit{Step 2: Checking $(Con 2)$. }  Let $\beta$ be a $\sigma(X^m_s;s\in[0,T])$-stopping time such that $\beta+\delta_0\leq T$. Set $m_0=\frac{1}{2}$ w.l.o.g.. Then for all $0<m\leq m_0$, one has $\frac{1}{2}\leq \gamma<1$. 
Similar to \eqref{diff}, one has
	{\small 	\begin{align}
		&\OX_{\beta+\delta}^m-\OX_{\beta}^m\notag\\
		&=\frac{m}{\gamma}(e^{-\frac{\gamma}{m}\beta}-e^{-\frac{\gamma}{m}(\beta+\delta)})\OV_0 \notag\\
		&\quad +\frac{\lambda_2}{\gamma}\int_0^\beta (e^{-\frac{\gamma}{m}(\beta-s)}-e^{-\frac{\gamma}{m}(\beta+\delta-s)}) (Y_s^{\alpha}(\bar\rho^m)-\OX_s^m)ds
		+\frac{\lambda_2}{\gamma}\int_\beta^{\beta+\delta} (1-e^{-\frac{\gamma}{m}(\beta+\delta-s)}) (Y_s^{\alpha}(\bar\rho^m)-\OX_s^m)ds\notag\\
		& \quad +\frac{\sigma_2}{\gamma}\int_0^\beta (e^{-\frac{\gamma}{m}(\beta-s)}-e^{-\frac{\gamma}{m}(\beta+\delta-s)})D(Y_s^{\alpha}(\bar\rho^m)-\OX_s^m)dB_s^2+\frac{\sigma_2}{\gamma}\int_\beta^{\beta+\delta} (1-e^{-\frac{\gamma}{m}(\beta+\delta-s)})D(Y_s^{\alpha}(\bar\rho^m)-\OX_s^m)dB_s^2\notag \\
			&\quad +\frac{\lambda_1}{\gamma}\int_0^\beta (e^{-\frac{\gamma}{m}(\beta-s)}-e^{-\frac{\gamma}{m}(\beta+\delta-s)}) (\OY_s^m-\OX_s^m)ds
		+\frac{\lambda_1}{\gamma}\int_\beta^{\beta+\delta} (1-e^{-\frac{\gamma}{m}(\beta+\delta-s)}) (\OY_s^m-\OX_s^m)ds\notag\\
		& \quad +\frac{\sigma_1}{\gamma}\int_0^\beta (e^{-\frac{\gamma}{m}(\beta-s)}-e^{-\frac{\gamma}{m}(\beta+\delta-s)})D(\OY_s^m-\OX_s^m)dB_s^1+\frac{\sigma_1}{\gamma}\int_\beta^{\beta+\delta} (1-e^{-\frac{\gamma}{m}(\beta+\delta-s)})D(\OY_s^m-\OX_s^m)dB_s^1\,.
		\end{align}}
	Using the estimate \eqref{uniform} it follows from the same computations as in Theorem \ref{thmtight} that
\begin{align}\label{diffX}
		&\EE[|\OX_{\beta+\delta}^m-\OX_{\beta}^m|^2]\leq C\left(\delta^{\frac{1}{4}} + \delta^2\right)\,,
		\end{align}
		where $C $ depends only on $\EE[|\OX_0|^4+|\OY_0|^4+|\OV_0|^4],\lambda_1,\sigma_2,\lambda_2,\sigma_2,T,C_{\alpha,\TE}$, and $\nu$.
		
		Having a look at 
			\begin{align*}
		\OY_{\beta+\delta}^{m}-\OY_{\beta}^{m}=\nu\int_\beta^{\beta+\delta}\left(\OX_{s}^{m}-\OY_{s}^{m}\right) S^{\beta}\left(\OX_{s}^{m}, \OY_{s}^{m}\right) ds,
		\end{align*}
		 we have
		\begin{equation*}
		|\OY_{\beta+\delta}^{m}-\OY_{\beta}^{m}|^2\leq \nu^2 \delta\int_0^{T}|\OX_{s}^{m}-\OY_{s}^{m}|^2d t\,.
		\end{equation*}
		By estimate \eqref{uniform}, we have
		\begin{equation}
		\EE[|\OY_{\beta+\delta}^{m}-\OY_{\beta}^{m}|^2]\leq  \nu^2 \delta \int_0^{T}\EE[|\OX_{s}^{m}-\OY_{s}^{m}|^4]^{\frac{1}{2}}d t\leq C\delta \,,
		\end{equation}
		where $C $ depends on $\EE[|\OX_0|^4+|\OY_0|^4+|\OV_0|^4],\lambda_1,\sigma_2,\lambda_2,\sigma_2,T,C_{\alpha,\TE}$ and $\nu$. This together with \eqref{diffX} justifies  $(Con 2)$.
	\end{proof}

Let us recall
\begin{align}\label{OXm}
\OX_t^m&=\OX_0+\frac{m}{\gamma}(1-e^{-\frac{\gamma}{m}t})\OV_0
+\frac{\lambda_1}{\gamma}\int_0^t(1-e^{-\frac{\gamma}{m}(t-s)}) (\OY_s^m-\OX_s^m)ds +\frac{\sigma_1}{\gamma}\int_0^t(1-e^{-\frac{\gamma}{m}(t-s)}) D(\OY_s^m-\OX_s^m)dB_s^1\notag\\
&\quad +\frac{\lambda_2}{\gamma}\int_0^t(1-e^{-\frac{\gamma}{m}(t-s)}) (Y_s^{\alpha}(\bar\rho^m)-\OX_s^m)ds +\frac{\sigma_2}{\gamma}\int_0^t(1-e^{-\frac{\gamma}{m}(t-s)}) D(Y_s^{\alpha}(\bar\rho^m)-\OX_s^m)dB_s^2
\end{align}
and
	\begin{align}\label{OYm}
\OY_{t}^{m}=\OY_0+\nu\int_0^t\left(\OX_{s}^{m}-\OY_{s}^{m}\right) S^{\beta}\left(\OX_{s}^{m}, \OY_{s}^{m}\right) d s\,.
\end{align}
Then following the lines of the proof in Theorem \ref{thmzero-limit}, one can easily obtain the following theorem.
\begin{thm}[Zero-inertia limit]\label{thmzero-limit1} 
	Let Assumption \ref{asum} hold and $(\OX_t^m,\OY_t^m)_{t\in[0,T]}$ satisfy the system \eqref{OXm}--\eqref{OYm}. Then as $m\rightarrow 0^+$, the sequence of  stochastic processes $\{(\OX^m,\OY^m)\}_{0< m\leq 1}$  converge weakly to $(\OX,\OY)$  which is the unique solution to the following coupled SDE:
\begin{align*}
\OX_t&=\OX_0
+\lambda_1\int_0^t (\OY_s-\OX_s)ds +\sigma_1\int_0^tD(\OY_s-\OX_s)dB_s^1+\lambda_2\int_0^t (Y_s^{\alpha}(\bar\rho)-\OX_s)ds +\sigma_2\int_0^t D(Y_s^{\alpha}(\bar\rho)-\OX_s)dB_s^2,
\\
\OY_{t}&=\OY_0+\nu\int_0^t\left(\OX_{s}-\OY_{s}\right) S^{\beta}\left(\OX_{s}, \OY_{s}\right) d s\,.
\end{align*}
Moreover,  we have the following convergences
\begin{align}\label{est-convegence-m-xy}
\EE\left[\max_{t\in[0,T]} \left(\big|\OX^m_t-\OX_t\big|^2
+\big|\OY^m_t-\OY_t\big|^2\right)
\right]\leq C \, \sqrt{m},\quad
\max_{t\in[0,T]} \EE\left[\big|\OX^m_t-\OX_t\big|^2
+\big|\OY^m_t-\OY_t\big|^2
\right]\leq C \, {m}.
\end{align}
where the constants $C$s depend  on $\EE[|\OX_0|^4+|\OY_0|^4+|\OV_0|^4],\lambda_1,\sigma_2,\lambda_2,\sigma_2,\beta,T,C_{\alpha,\TE},d$, and $\nu$.
\end{thm}

\section{Numerical examples on the zero-inertia limit}\label{numerics}
We conclude this paper with a few instructive numerical experiments on validating the zero-inertia limit. We will focus on the mono-dimensional case since it allows us to see more clearly how the distribution of particles evolves in time depending on the inertia parameter $m$, and hence show the zero-inertia limit. Different benchmark functions have been used and tested, but we will report here the case of the Ackley function shown in Figure \ref{fig:ackley1D}. Following the same structure of the paper, we will first analyze the case without memory effect and then we will generalize as in Section \ref{section_memory} to the case with memory. Extensive discussions on other numerical implementations and experiments are presented in \cite{grassi2020particle}.

\begin{figure}[h]
\includegraphics[scale=0.3]{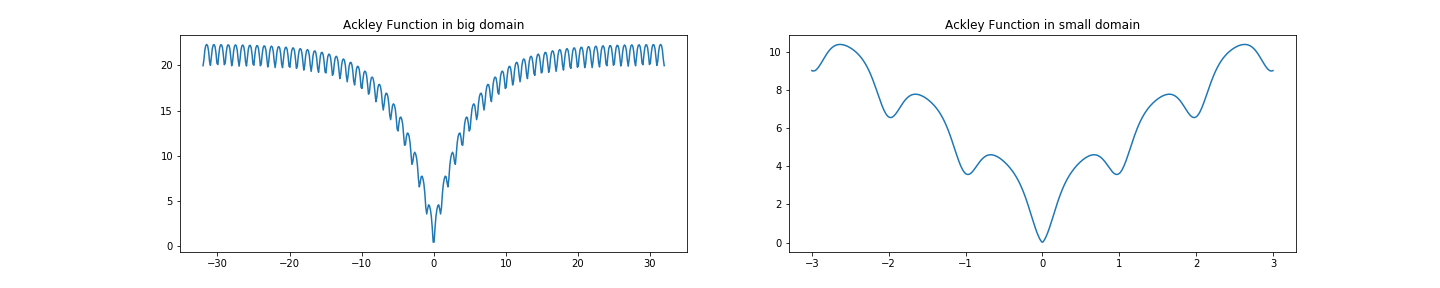}\,
\caption{Ackley function in a big (left) and small (right) domain with its many local minima.}
\label{fig:ackley1D}
\end{figure}

\subsection{Small inertia limit without memory}
Given the system of stochastic differential equations in \eqref{PSO}, the particle system can be solved by using a semi-implicit discretization scheme
\begin{align}\label{PSO_discrete}
\begin{cases}
X_{n+1}^{i,m}= X_n^{i,m} +\Delta t V_{n+1}^{i,m}, \\
V_{n+1}^{i,m}=\frac{m}{m+\gamma \Delta t}V_n^{i,m}+\frac{\lambda \Delta t}{m + \gamma \Delta t}(X_n^{\alpha,m}-X_n^{i,m})+\frac{\sigma \sqrt{\Delta t}}{m + \gamma \Delta t}D(X_n^{\alpha,m}-X_n^{i,m})\theta_n^i,\quad i=1,\cdots,N\,,
\end{cases}
\end{align}
where $X_n^{i,m}$ and $V_{n}^{i,m}$ are, respectively, the position and velocity of the i-th particle at the discrete time $n\Delta t$ with $\Delta t$ being the time discretization, and the diagonal matrix $D(X_n^{\alpha,m}-X_n^{i,m})$ simply coincides with $X_n^{\alpha,m}-X_n^{i,m}$ as we are considering the mono-dimensional case. Moreover, $X_n^{\alpha,m}$ is defined as in \eqref{regularizer} and $\theta^i_n \sim \mathcal N(0,1)$ $\forall i,n$.
We compare this particle system with the CBO dynamic of the form \eqref{MVCBO}, which can be solved using the Euler-Maruyama scheme
\begin{equation}\label{CBO_discrete}
X^i_{n+1} = X_{n}^i + \Delta t \lambda (X_n^{\alpha} - X_n^i) + \sqrt{\Delta t } \sigma (X_n^{\alpha}-X_n^i) \theta_n^i\, .
\end{equation}
As already mentioned, we consider the minimization of the Ackley function with minimum at $x=0$ and, starting from the same initial distribution of particles, we solve the PSO system \eqref{PSO_discrete} for different inertia values. Then, we compare the evolution of the distribution of particles with the one of the particles moving according to the CBO system \eqref{CBO_discrete}. In order to be able to compare the results, we fix the parameters $\lambda = 1$, $\sigma = \frac{1}{\sqrt{3}}$ and $\alpha = 30$, while $\theta_n^i $ are sampled from $\mathcal N(0,1)$ and fixed for each $i= 1,..., N$ and $n \in [0,T/\Delta t]$. Moreover, $T$ is set to 1 and the time discretization is $\Delta t = 0.01$, with a total number of particles $N = 10^4$.\\
\begin{figure}
\includegraphics[scale=0.3]{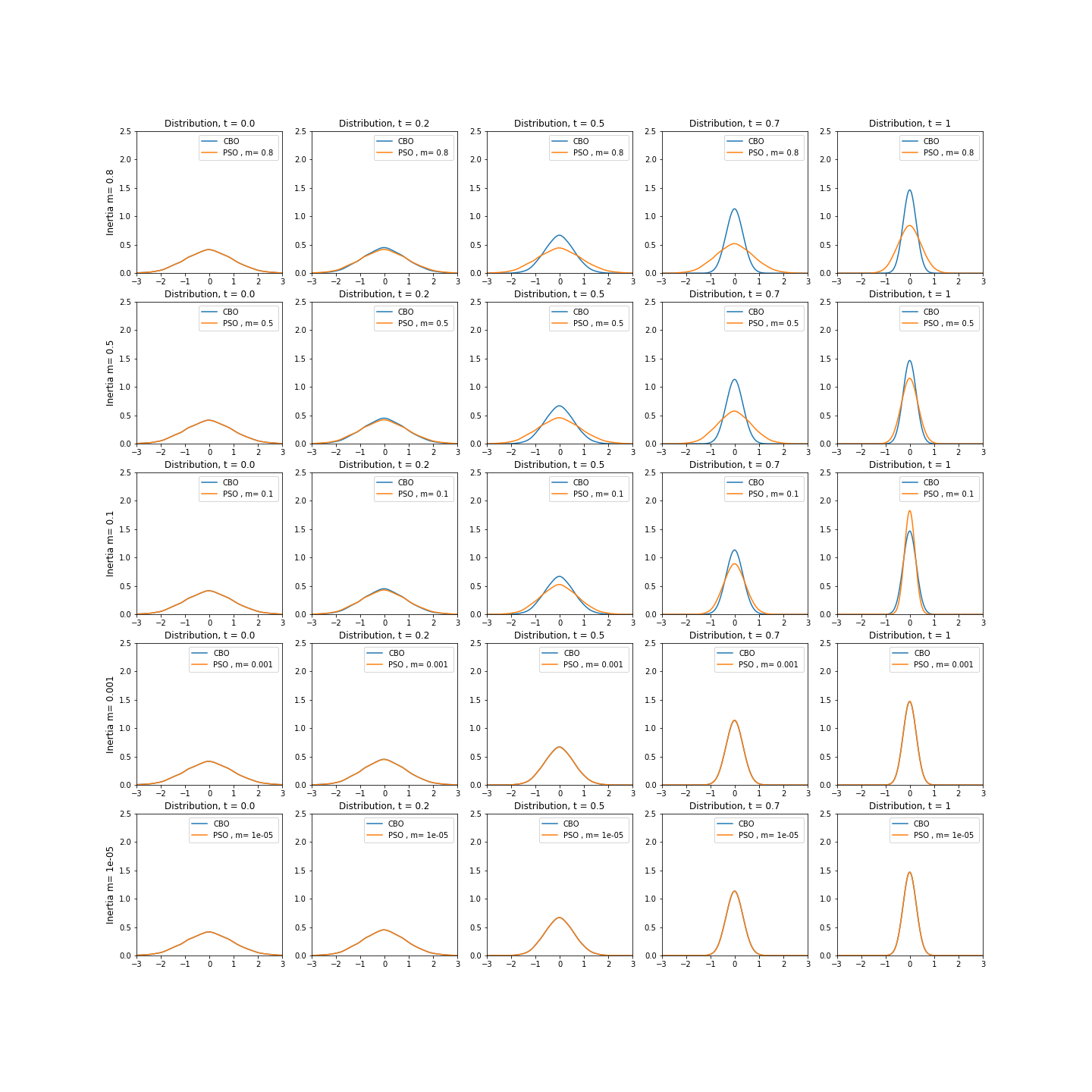}\,
\caption{Comparison of the CBO \eqref{CBO_discrete} and PSO \eqref{PSO_discrete} dynamics for different inertia values (which are changing over the rows) and at many time steps (changing over columns), starting from a normal distribution.}
\label{fig:small_inertia_limit_gaussian}
\end{figure}
Figure \ref{fig:small_inertia_limit_gaussian} shows in each row the evolution of the CBO distribution and the one of the PSO system with $m$ fixed that is decreasing over rows. The initial particles are always sampled from the same distribution, which is in this case a Gaussian centered in $0$ with variance $1$. Clearly, the PSO system with $m= 0.8$ leads to the correct minimum in 0 at the final time step $t=1$, but the distribution of the particles is different from the one of the CBO. While, for any $t \in [0,T]$, if the inertia value is decreased to $0.1$, or even to $0.001$, the two distributions, namely the one obtained via CBO and the PSO one, are indistinguishable, as the last two rows of Figure \ref{fig:small_inertia_limit_gaussian} show.\\ These considerations are confirmed in Figure \ref{fig:comparison_distributions} where we compare the distributions obtained in Figure \ref{fig:small_inertia_limit_gaussian} using the Wasserstein 2 distance between the CBO distribution and the PSO distribution. On the left of Figure \ref{fig:comparison_distributions}, the Wasserstein distance is plotted for each time step. Moreover, since we want to show the influence of the inertia parameter, we take the mean value of the Wasserstein distance over all time steps and plot it as a function of the inertia values. This is shown on the right of Figure \ref{fig:comparison_distributions} where we also add the mean value of the Kullback-Leibler divergence since the latter is a well-known measure used to compare distributions, especially in statistics.
\begin{figure}
\includegraphics[scale=0.3]{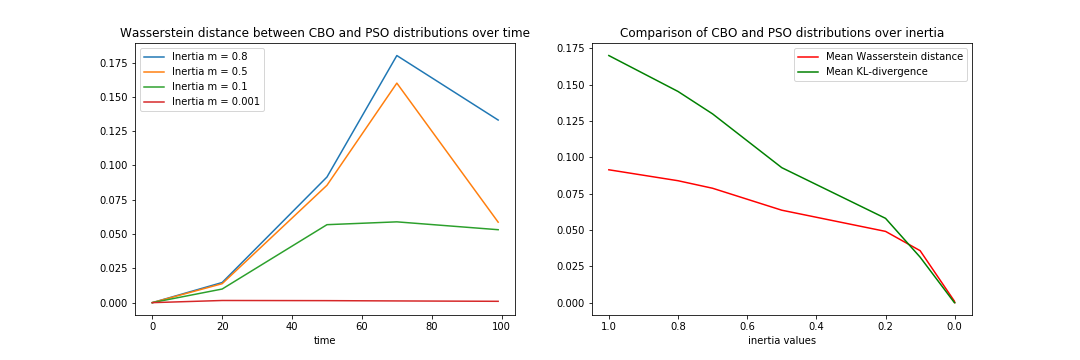}\,
\caption{Left: Wasserstein distance over time ; Right: Wasserstein distance and Kullback-Leibler divergence (in mean) over the inertia values.}
\label{fig:comparison_distributions}
\end{figure}
Moreover, since it is necessary to start with an initial distribution that is close to the global minimizer, we also try to see what happens when the initial distribution is a uniform distribution between $-3$ and $3$ and compare the evolution of its particles according to the CBO and PSO dynamics, with varying inertia parameters. The result is shown in Figure \ref{fig:small_inertia_limit_uniform}. In this case, the difference between the CBO distribution and the one of the PSO dynamics is way higher in the case of big inertia value, but, as before, goes to zero as soon as $m$ converges to $0$.
\begin{figure}
\includegraphics[scale=0.3]{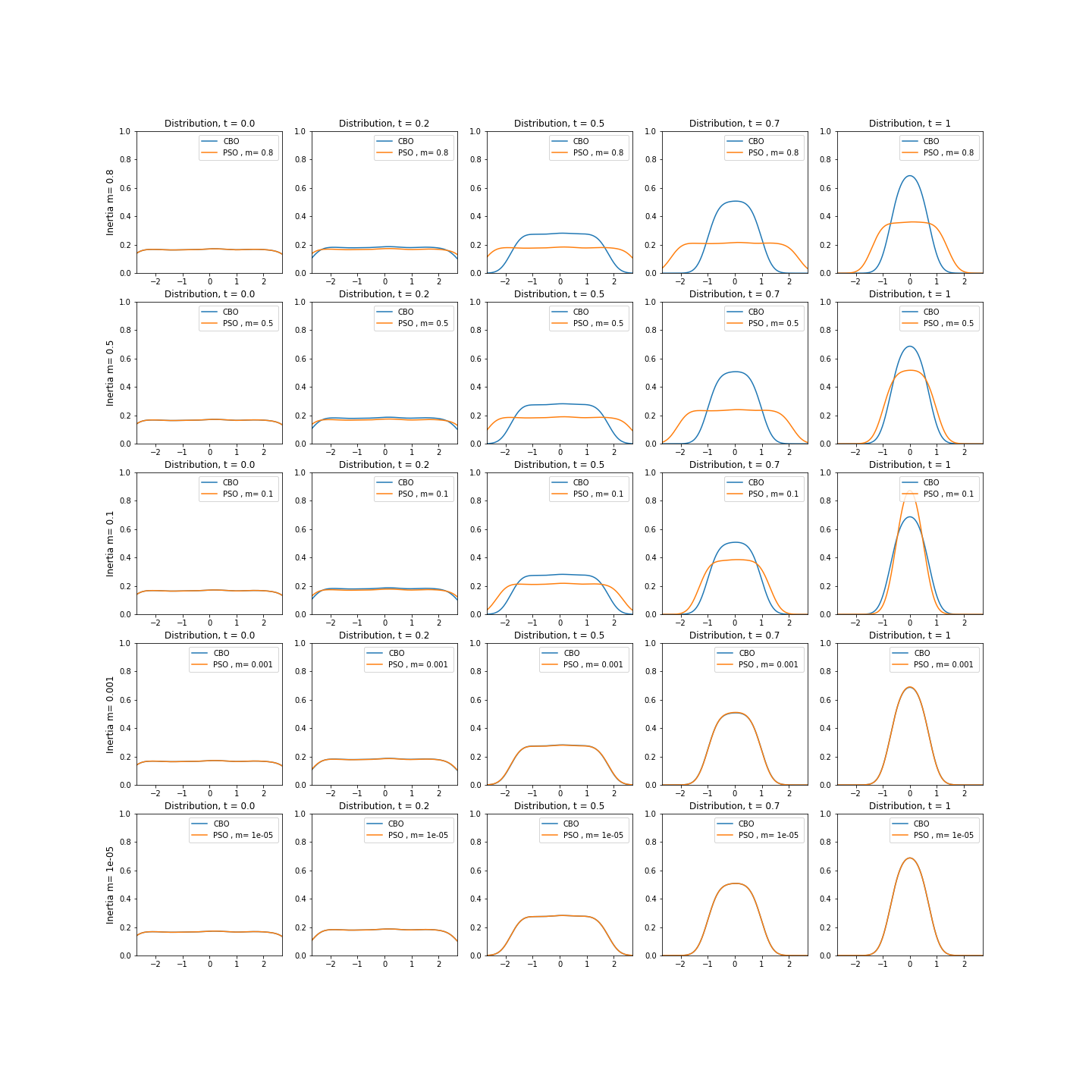}\,
\caption{Evolution of an initial uniform distribution according to CBO \eqref{CBO_discrete} and PSO \eqref{PSO_discrete} dynamics and their comparison for different time steps (on the columns) and different inertia values (on the rows).}
\label{fig:small_inertia_limit_uniform}
\end{figure}

\subsection{Small inertia limit with memory effect}
The PSO model which involves the memory of the local and global best positions, underlying \eqref{mXeq}--\eqref{mVeq}, can similarly be solved via
{\small \begin{align}\label{PSO_memory_discrete}
\begin{cases}
X_{n+1}^{i,m} = X_n^{i,m} +\Delta t V_{n+1}^{i,m}, \quad i=1,\cdots,N\\
Y_{n+1}^{i,m} =Y_{n}^{i,m}+\nu \Delta t (X_{n+1}^{i,m}-Y_n^{i,m}) S^{\beta}(X_{n+1}^{i,m}, Y_n^{i,m}), \\
V_{n+1}^{i,m}= \frac{m}{m+\gamma \Delta t}V_n^{i,m}+\frac{\lambda_1 \Delta t}{m + \gamma \Delta t}(Y_n^{i,m}-X_n^{i,m})+\frac{\lambda_2 \Delta t}{m + \gamma \Delta t}(Y_n^{\alpha,m}-X_n^{i,m}) \\
\qquad \quad +\frac{\sigma_1 \sqrt{\Delta t}}{m + \gamma \Delta t}D(Y_n^{i,m}-X_n^{i,m})\theta_n^{1,i} +\frac{\sigma_2 \sqrt{\Delta t}}{m + \gamma \Delta t}D(Y_n^{\alpha,m}-X_n^{i,m})\theta_n^{2,i}\,,
\end{cases}
\end{align}}
where $Y_n^{i,m}$ is the local best that the i-th particle has memory of, and $ Y_n^{\alpha,m}$ is the regularized global best, defined as in \eqref{global_best_regular}. Clearly, the corresponding CBO dynamics is the following
\begin{align}\label{CBO_memory_discrete}
\begin{cases}
X_n^{i} = X_n^i + \lambda_1 \Delta t (Y_n^i-X_n^i) + \lambda_2 \Delta t (Y_n^{\alpha}-X_n^i) + \sigma_1 \sqrt{\Delta t} D(Y_n^i-X_n^i) \theta_n^{1,i} + \sigma_2 \sqrt{\Delta t} (Y_n^{\alpha} -X_n^i) \theta_n^{2,i} \\
Y_n^i = Y_n^i + \nu \Delta t (X_n^i-Y_n^i) S^{\beta}(X_n^i, Y_n^i)
\end{cases}
\end{align}
Once again, since we want to show the convergence of the PSO distribution obtained from \eqref{PSO_memory_discrete} with a small inertia value to the one attained via the CBO system \eqref{CBO_memory_discrete}, we need to set some of the parameters to the same values in order to be able to compare the results. Their values are the following
\begin{equation}
\lambda_1 = \lambda_2 = 1 \quad \sigma_1 = \sigma_2 = \frac{1}{\sqrt{3}} \quad \alpha = 30 \quad \beta=30 \quad \nu = \frac{1}{2}
\end{equation}
and, as before, the effect of the Brownian motion leads to $\theta_n^{1,i}, \theta_n^{2,i}$ which are sampled from a normal distribution and set to a fixed value $\forall i= 1, ..,N$ and $\forall n \in [0,T / \Delta t]$. The time discretization and the number of particles are set to the same values as in the case without memory, namely $T = 1$, $\Delta t = 0.01$, and $N=10^4$. The difference with the previous case is that now the distribution of which we want to show convergence, is actually a function of both the particles' position and their local best and, as such, it is bi-dimensional. The result of the evolution of the particles according to the CBO dynamics \eqref{CBO_memory_discrete} is shown in the first row of Figure \ref{fig:small_inertia_limit_with_memory}, while the second row represents the evolution according to the PSO dynamics \eqref{PSO_memory_discrete} with a big inertia parameter, e.g. $m= 0.8$. This is compared with the evolution presented in the third row in which the inertia is set to a very low value, e.g. $m= 0.001$.\\
\begin{figure}
\includegraphics[scale=0.3]{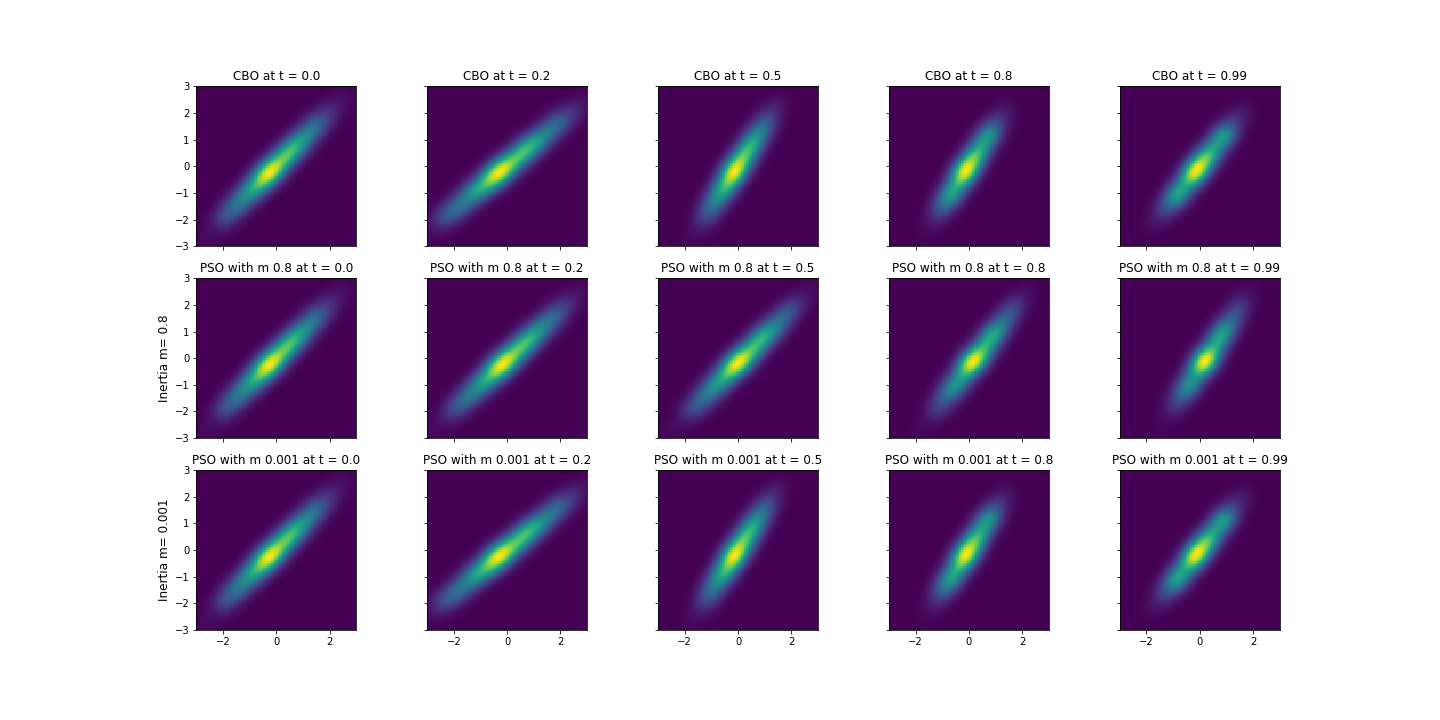}\,
\caption{First Row: evolution in time of the initial gaussian distribution according to CBO dynamics \eqref{CBO_memory_discrete}; Second Row: evolution in time of the initial Gaussian distribution according to PSO dynamics \eqref{PSO_memory_discrete} with $m = 0.8$; Third Row: evolution produced by the PSO dynamics \eqref{PSO_memory_discrete} with $m = 0.001$.}
\label{fig:small_inertia_limit_with_memory}
\end{figure}
In the case without memory, it is easy to see the convergence at each time step of the PSO system with small inertia to the CBO, but it's not as clear now that the distribution we are interested in is bi-dimensional. To be able to compare the evolution and to check how similar the distributions are at every time step, we show in Figure \ref{fig:comparison_distribution_2D} different plots: we look at the particles' mean positions (left) and their local best (center) for each time step and, in both cases, show their standard deviation as a colored area around the mean. It is interesting to see how the particles are moving and where they are attracted to, especially because the pattern of the PSO dynamics with small inertia is the same as the one of CBO. Finally, on the right plot of Figure \ref{fig:comparison_distribution_2D}, we show how both our similarity measures, namely the Wasserstein distance and the Kullback-Leibler divergence, are decreasing along  with the inertia parameter, validating the small inertia limit also in the general case with memory.
\begin{figure}
\includegraphics[scale=0.3]{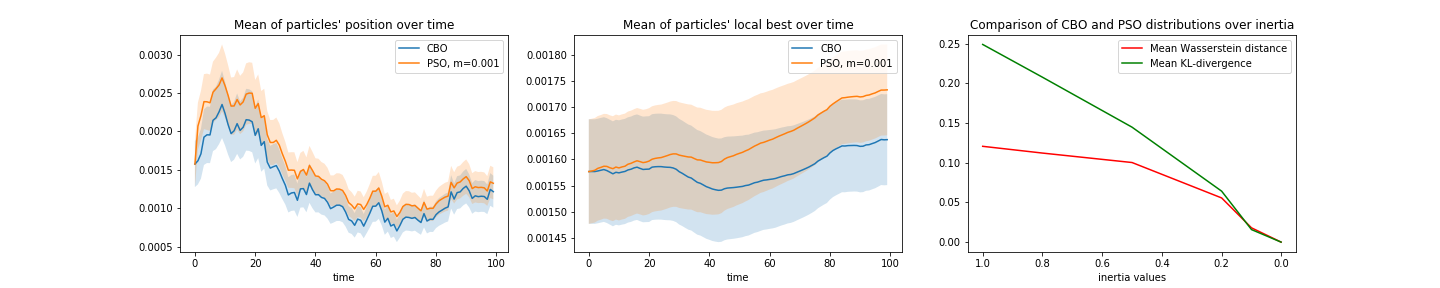}\,
\caption{Left: mean particles' position with their standard deviation; Center: mean particles' local best with the standard deviation too; Right: plot of the Wasserstein distance and Kullback–Leibler divergence between the distribution obtained via CBO dynamics \eqref{CBO_memory_discrete} and the one obtained through PSO system \eqref{PSO_memory_discrete} with small inertia value, i.e. $0.001$.}
\label{fig:comparison_distribution_2D}
\end{figure}

Finally, we want to numerically investigate the convergence in path space given in Theorem \ref{thmzero-limit}  and Remark \ref{remark_convergence}. Figure \ref{fig:convergence_no_memory} shows on the left the trajectory of a random particle which is moving according to the CBO system \eqref{CBO_discrete} and the PSO system \eqref{PSO_discrete} with different inertia values. Using all the $N = 10^4$ particles, we take 
\begin{equation}\label{convergence_with_L2_norm}
\lim_{k\rightarrow \infty} \frac{1}{N} \sum_{i=1}^{N} \max_{n \in [0,T/\Delta t]}\abs{X^{i,m}_n-X^{i}_n}^2 ,
\end{equation}
namely the squared $L^2$-norm of the vector containing the sup-norm in time of the distances between the CBO and PSO solutions. On the right of Figure \ref{fig:convergence_no_memory}, the numerical verification of formula \eqref{strong_convergence} is presented: for every inertia value $m$, the quantity $\max_{n \in [0,T/\Delta t]} \abs{X^{i_0,m}_t-X^{i_0}_n}^2$ is plotted, where $i_0$ indicates the same random particles whose trajectories are plotted on the left of Figure  \ref{fig:convergence_no_memory}. Moreover, the quantity in \eqref{convergence_with_L2_norm} is also plotted for every inertia value $m$ and the convergence to the limit $0$ is clear.

\begin{figure}
\includegraphics[scale=0.3]{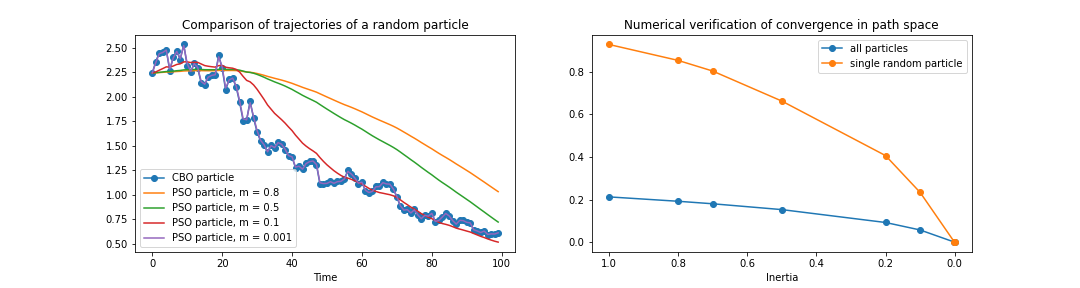}\,
\caption{Left: given a random particle, its trajectory according to CBO \eqref{CBO_discrete} is compared with the ones obtained via PSO \eqref{CBO_discrete} with different inertia values ; Right: verification of limit \eqref{convergence_with_L2_norm}, i.e. plot of $\frac{1}{N} \sum_{i=1}^{N} \max_{n \in [0,T/\Delta t]} \abs{X^{i,m}_n-X^{i}_n}^2$ for every decreasing value of $m$ (in blue) and of the quantity $\max_{n \in [0,T/\Delta t]} \abs{X^{i_0,m}_n-X^{i_0}_n}^2$ for the same random particles considered on the left (orange color).}
\label{fig:convergence_no_memory}
\end{figure}

Both the plots in Figure \ref{fig:convergence_no_memory} represent the particles moving according to the discrete systems \eqref{PSO_discrete},  \eqref{CBO_discrete} in the case without memory effect. In the more general case with memory, the systems \eqref{PSO_discrete},  \eqref{CBO_discrete} become the ones in \eqref{PSO_memory_discrete} and  \eqref{CBO_memory_discrete}, but the convergence in path space still holds and it's numerically shown in Figure \ref{fig:convergence_memory}. Here, not only the trajectories of a random particle are plotted (on the left), but also the evolution of the local bests according to both the discretized systems and with different inertia values are shown in the central image of Figure \ref{fig:convergence_memory}. Once again, even in the case with memory effect, the strong convergence in path space is shown for the random particle previously considered and also for all of the particles in terms of the quantity $\frac{1}{N} \sum_{i=1}^{N} \max_{n \in [0,T/\Delta t]} \left(\abs{X^{i,m}_n-X^{i}_n}^2 + \abs{Y^{i,m}_n-Y^{i}_n}^2 \right)$, on the right of Figure \ref{fig:convergence_memory}.

\begin{figure}
\includegraphics[scale=0.3]{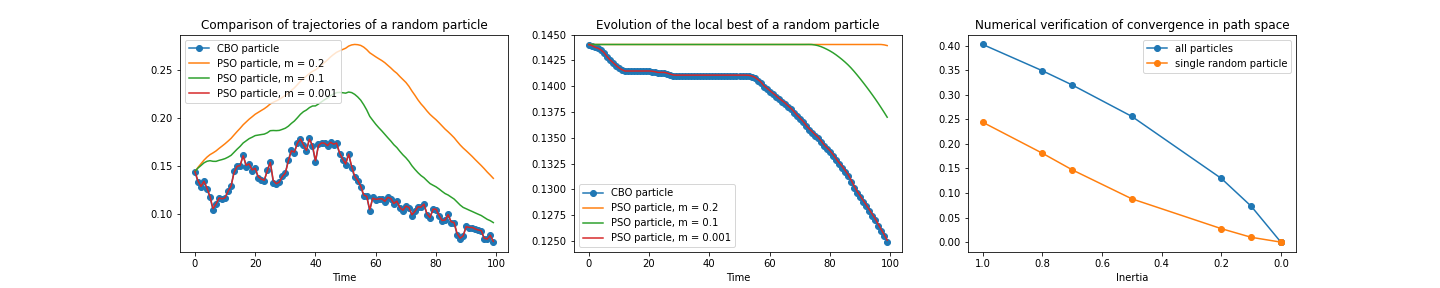}\,
\caption{Left: random particle's trajectories according to CBO \eqref{CBO_memory_discrete} and PSO \eqref{PSO_memory_discrete} with decreasing inertia values; Center: random particle's memory of the local best according to CBO \eqref{CBO_memory_discrete} and PSO \eqref{PSO_memory_discrete} with decreasing inertia values; Right: verification of limit of $\frac{1}{N} \sum_{i=1}^{N} \max_{n \in [0,T/\Delta t]} \left(\abs{X^{i,m}_n-X^{i}_n}^2 + \abs{Y^{i,m}_n-Y^{i}_n}^2 \right)$ for every decreasing value of $m$ (in blue) and of the quantity $\max_{n \in [0,T/\Delta t]} \left( \abs{X^{i_0,m}_n-X^{i_0}_n}^2+  \abs{Y^{i_0,m}_n-Y^{i_0}_n}^2 \right)$ for the same random particles considered on the left (orange color).}
\label{fig:convergence_memory}
\end{figure}

 Figure \ref{fig:comparison} shows a comparison between $L^2$-estimate of the form \eqref{convergence_with_L2_norm} and the one of the following form
\begin{equation}\label{convergence_sup_outside}
\lim_{k\rightarrow \infty} \max_{n \in [0,T/\Delta t]} \frac{1}{N} \sum_{i=1}^{N} \abs{X^{i,m}_n-X^{i}_n}^2 ,
\end{equation}
in order to check their convergence rate. On the left, both the quantities in formula \eqref{convergence_with_L2_norm} and \eqref{convergence_sup_outside} are plotted for every inertia value in the case without memory effect. While on the right, the corresponding formulas (indicated in the caption of Figure \ref{fig:comparison}) are plotted for the case with memory effect. In both cases, our numerical experiments indicate that just like the rate of convergence given in \eqref{convergence-l2} (resp. \eqref{est-convegence-m-xy}), the $L^2$-estimate of the uniform norm of the paths of $\OX^{m}-\OX$ (resp. $(\OX^{m}-\OX,\,\, \OY^m_t-\OY)$)  is on the order of $\sqrt{m}$ while the $L^2$-estimate for each time $t\in[0,T]$ has order of $m$.

\begin{figure}
\includegraphics[scale=0.3]{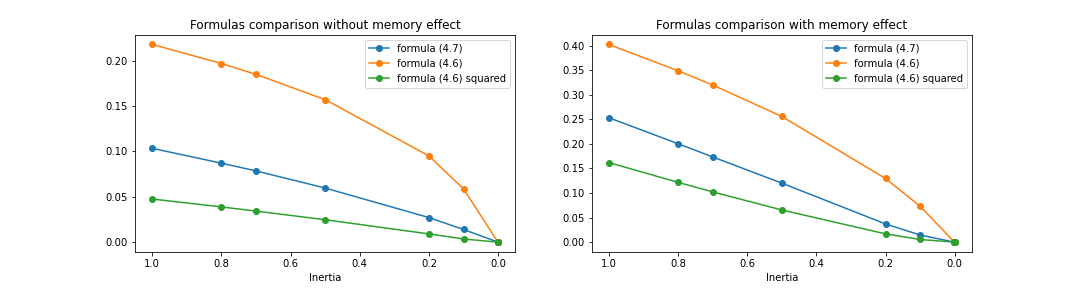}\,
\caption{Left: case without memory, comparison of $\frac{1}{N} \sum_{i=1}^{N} \max_{n \in [0,T/\Delta t]} \abs{X^{i,m}_n-X^{i}_n}^2$ (orange) and  $\max_{n \in [0,T/\Delta t]} \frac{1}{N} \sum_{i=1}^{N}  \abs{X^{i,m}_n-X^{i}_n}^2$ (blue) for every inertia value, where the curve $\left(\frac{1}{N} \sum_{i=1}^{N} \max_{n \in [0,T/\Delta t]} \abs{X^{i,m}_n-X^{i}_n}^2 \right)^2$ (green) is also plotted as a function of $m$;  Right: case with memory, comparison of quantity $\frac{1}{N} \sum_{i=1}^{N} \max_{n \in [0,T/\Delta t]} \left(\abs{X^{i,m}_n-X^{i}_n}^2 + \abs{Y^{i,m}_n-Y^{i}_n}^2 \right)$ (orange) with $\max_{n \in [0,T/\Delta t]} \frac{1}{N} \sum_{i=1}^{N}  \left(\abs{X^{i,m}_n-X^{i}_n}^2 + \abs{Y^{i,m}_n-Y^{i}_n}^2 \right)$ (blue) for every decreasing value of $m$, where the square $\left(\frac{1}{N} \sum_{i=1}^{N} \max_{n \in [0,T/\Delta t]} \left(\abs{X^{i,m}_n-X^{i}_n}^2 + \abs{Y^{i,m}_n-Y^{i}_n}^2 \right) \right)^2$ (green) is also plotted.}
\label{fig:comparison}
\end{figure}

\bibliographystyle{amsxport}
\bibliography{PSO}


\end{document}